\documentclass[a4paper,12pt]{preprint}
\usepackage[margin=1.3in]{geometry}  
\usepackage[full]{textcomp}
\usepackage[osf]{newtxtext}

\usepackage{mathalfa}
\usepackage{microtype}

\usepackage{booktabs}
\usepackage{times}

\usepackage{enumitem}
\usepackage{amsmath}
\usepackage{amsfonts} 
\usepackage{amssymb}             

\usepackage{comment}
\usepackage{breakurl}

\usepackage{mathrsfs}
\usepackage{booktabs}
\usepackage{mathtools}

\usepackage{tikz}
\usepackage{amsthm}                
\usepackage{mhequ}
\usepackage{hyperref}

\hypersetup{pdfstartview=}

\addtocontents{toc}{\protect\hypersetup{hidelinks}}

\DeclareSymbolFont{sfoperators}{OT1}{ptm}{m}{n}
\DeclareSymbolFontAlphabet{\mathsf}{sfoperators}

\makeatletter
\def\operator@font{\mathgroup\symsfoperators}
\makeatother

\numberwithin{equation}{section}

\newtheorem{thm}{Theorem}[section]
\newtheorem{defn}[thm]{Definition}
\newtheorem{lem}[thm]{Lemma}
\newtheorem{prop}[thm]{Proposition}

\newtheorem{assumption}[thm]{Assumption}

\makeatletter
\def\th@newremark{\th@remark\thm@headfont{\bfseries}}

\def\bdiamond{\mathop{\mathpalette\bdi@mond\relax}}
\newcommand\bdi@mond[2]{%
	\vcenter{\hbox{\m@th
			\scalebox{\ifx#1\displaystyle 2.6\else1.8\fi}{$#1\diamond$}%
	}}%
}

\def\bDiamond{\mathop{\mathpalette\bDi@mond\relax}}
\newcommand\bDi@mond[2]{%
	\vcenter{\hbox{\m@th
			\scalebox{\ifx#1\displaystyle 2.6\else1.2\fi}{$#1\Diamond$}%
	}}%
}
\makeatletter

\theoremstyle{newremark}

\newtheorem{rmk}[thm]{Remark}

\newtheorem{rem}[thm]{Remark}

\definecolor{darkgreen}{rgb}{0.1,0.7,0.1}
\definecolor{darkred}{rgb}{0.7,0.1,0.1}
\definecolor{darkblue}{rgb}{0,0,0.7}
\newcommand{\PP}{\mathbb{P}}     

\newcommand{\RR}{\mathbb{R}}      

\newcommand{\fF}{\mathcal{F}}

\newcommand{\hH}{\mathcal{H}}
\newcommand{\iI}{\mathcal{I}}

\newcommand{\nN}{\mathcal{N}}

\newcommand{\xX}{\mathcal{X}}
\newcommand{\yY}{\mathcal{Y}}

\newcommand{\fG}{\mathfrak{G}}

\newcommand{\fg}{\mathfrak{g}}

\makeatletter 
\newcommand{\cov}{{\operator@font cov}}
\newcommand{\var}{{\operator@font var}}
\newcommand{\corr}{{\operator@font corr}}
\newcommand{\diam}{{\operator@font diam}}
\newcommand{\Av}{{\operator@font Av}}
\newcommand{\trig}{{\operator@font trig}}
\newcommand{\Enh}{{\operator@font Enh}}
\newcommand{\EEnh}{\overline {\operator@font Enh}}
\makeatother

\newcommand{\1}{\mathbf{1}}

\renewcommand{\k}{\mathbf{k}}

\newcommand{\eps}{\varepsilon}

\colorlet{symbols}{blue!90!black}
\colorlet{testcolor}{green!60!black}

\def\${|\!|\!|}

\usetikzlibrary{shapes.misc}
\usetikzlibrary{shapes.symbols}
\usetikzlibrary{snakes}
\usetikzlibrary{decorations}
\usetikzlibrary{decorations.markings}

\def\drawx{\draw[-,solid] (-3pt,-3pt) -- (3pt,3pt);\draw[-,solid] (-3pt,3pt) -- (3pt,-3pt);}
\tikzset{
	root/.style={circle,fill=testcolor,inner sep=0pt, minimum size=2mm},
	dot/.style={circle,fill=black,inner sep=0pt, minimum size=1mm},
	edot/.style={circle,fill=black,inner sep=0pt, minimum size=1mm},
	odot/.style={circle,draw=black,inner sep=0pt, minimum size=1mm},
	var/.style={circle,fill=black!10,draw=black,inner sep=0pt, minimum size=
	2mm},
    svar/.style={circle,fill=black!10,draw=black,inner sep=0pt, minimum size=
	1.5mm},
    noise0/.style={rectangle,draw=symbols,fill=white,inner sep=0pt, minimum size=1.5mm},
    noise1/.style={circle,draw=symbols,fill=white,inner sep=0pt, minimum size=1.5mm},
    noise2/.style={circle,draw=symbols,fill=symbols,inner sep=0pt, minimum size=1.5mm},
	dotred/.style={circle,fill=symbols!50,inner sep=0pt, minimum size=2mm},
	generic/.style={semithick,shorten >=1pt,shorten <=1pt},
	ageneric/.style={semithick},
	dist/.style={ultra thick,draw=testcolor,shorten >=1pt,shorten <=1pt},
	testfcn/.style={ultra thick,testcolor,shorten >=1pt,shorten <=1pt,<-},
	testfcnx/.style={ultra thick,testcolor,shorten >=1pt,shorten <=1pt,<-,
		postaction={decorate,decoration={markings,mark=at position 0.6 with {\drawx}}}},
	kepsilon/.style={semithick,shorten >=1pt,shorten <=1pt,densely dashed,->},
	kprimex/.style={semithick,shorten >=1pt,shorten <=1pt,densely dashed,->,
		postaction={decorate,decoration={markings,mark=at position 0.4 with {\drawx}}}},
	kernel/.style={semithick,shorten >=1pt,shorten <=1pt,->},
	akernel/.style={semithick,->},
	multx/.style={shorten >=1pt,shorten <=1pt,
		postaction={decorate,decoration={markings,mark=at position 0.5 with {\drawx}}}},
	kernelx/.style={semithick,shorten >=1pt,shorten <=1pt,->,
		postaction={decorate,decoration={markings,mark=at position 0.4 with {\drawx}}}},
	kernel1/.style={->,semithick,shorten >=1pt,shorten <=1pt,postaction={decorate,decoration={markings,mark=at position 0.45 with {\draw[-] (0,-0.1) -- (0,0.1);}}}},
	kernel2/.style={->,semithick,shorten >=1pt,shorten <=1pt,postaction={decorate,decoration={markings,mark=at position 0.45 with {\draw[-] (0.05,-0.1) -- (0.05,0.1);\draw[-] (-0.05,-0.1) -- (-0.05,0.1);}}}},
	kernelBig/.style={semithick,shorten >=1pt,shorten <=1pt,decorate, decoration={zigzag,amplitude=1.5pt,segment length = 3pt,pre length=2pt,post length=2pt}},
	gepsilon/.style={dotted,semithick,shorten >=1pt,shorten <=1pt},
	renorm/.style={shape=circle,fill=white,inner sep=1pt},
	labl/.style={shape=rectangle,fill=white,inner sep=1pt},
	xi/.style={circle,fill=symbols!10,draw=symbols,inner sep=0pt,minimum size=1.2mm},
	xix/.style={crosscircle,fill=symbols!10,draw=symbols,inner sep=0pt,minimum size=1.2mm},
	xib/.style={circle,fill=symbols!10,draw=symbols,inner sep=0pt,minimum size=1.6mm},
	xibx/.style={crosscircle,fill=symbols!10,draw=symbols,inner sep=0pt,minimum size=1.6mm},
	not/.style={circle,fill=symbols,draw=symbols,inner sep=0pt,minimum size=0.5mm},
	>=stealth,
  	highlight/.style={line width=7pt,blue,draw opacity=0.2,line cap=round,line join=round},
  	cover/.style={line width=7pt,blue,line cap=round,line join=round},
	smalldot/.style={circle,fill=symbols,draw=symbols, solid,inner sep=0pt,minimum size=0.5mm},
	}

\makeatletter
\def\DeclareSymbol#1#2#3{\expandafter\gdef\csname MH@symb@#1\endcsname{\tikz[baseline=#2,scale=0.15,draw=symbols]{#3}}\expandafter\gdef\csname MH@symb@#1s\endcsname{\scalebox{0.5}{\tikz[baseline=#2,scale=0.15,draw=symbols]{#3}}}}
\def\<#1>{\csname MH@symb@#1\endcsname}
\makeatother

\DeclareSymbol{Xi22}{0.5}{\draw (0,0) node[xi] {} -- (-1,1) node[not] {} -- (0,2) node[xi] {}; }

\DeclareSymbol{0'}{0}{\node at (0,0.7) [noise0] {}; }
\DeclareSymbol{1'}{0}{\node at (0,0.7) [noise1] {}; }
\DeclareSymbol{2'}{0}{\node at (0,0.7) [noise2] {}; }

\DeclareSymbol{0'0}{0}{\draw (0,-0.5) node[not] {} -- (0,1.5) node[noise0] {};}
\DeclareSymbol{1'0}{0}{\draw (0,-0.5) node[not] {} -- (0,1.5) node[noise1] {};}
\DeclareSymbol{2'0}{0}{\draw (0,-0.5) node[not] {} -- (0,1.5) node[noise2] {};}
\DeclareSymbol{1'1'}{0}{\draw (-1,1.5) node[noise1] {}  -- (0.5,0) node[noise1] {};}
\DeclareSymbol{2'1'}{0}{\draw (-1,1.5) node[noise2] {}  -- (0.5,0) node[noise1] {};}
\DeclareSymbol{2'1'0}{0}{\draw (1.5,2) node[noise2] {}  -- (0,0.5) node[noise1] {} -- (1.5,-1) node[not] {};}
\DeclareSymbol{2'2'0}{-2}{\draw (-1,1) node[noise2] {}  -- (0,-0.5) node[smalldot] {} -- (1,1) node[noise2] {};}
\DeclareSymbol{2'1'1'}{-1}{\draw (0,1.4) node[noise2] {}  -- (1.5,0.5) node[noise1] {} -- (0,-0.4) node[noise1] {};}
\DeclareSymbol{2'0'}{0}{\draw (-1,1.5) node[noise2] {}  -- (0.5,0) node[noise0] {};}
\DeclareSymbol{2'2'0'}{0}{\draw (-1,1.5) node[noise2] {}  -- (0,0) node[noise0] {} -- (1,1.5) node[noise2] {};}

\DeclareSymbol{0}{0}{\draw[white] (-.4,0) -- (.4,0); \draw (0,0)  -- (0,1.2) node[smalldot] {};}
\DeclareSymbol{1}{0}{\draw[white] (-.4,0) -- (.4,0); \draw (0,0)  -- (0,1.2) node[smalldot] {};}
\DeclareSymbol{2}{0}{\draw (-0.5,1.2) node[smalldot] {} -- (0,0) -- (0.5,1.2) node[smalldot] {};}
\DeclareSymbol{11}{0}{\draw (0,1.8) node[smalldot] {} -- (-0.7,0.9) -- (0,0) -- (0.7,1) node[smalldot] {};}
\DeclareSymbol{10}{0}{\draw (0,1.8) node[smalldot] {} -- (-0.8,0.9) -- (0,0);}
\DeclareSymbol{21}{0.7}{\draw (-1,1.8) node[smalldot] {} -- (-0.5,0.9); \draw (0,1.8) node[smalldot] {} -- (-0.5,0.9) -- (0,0) -- (0.5,0.9) node[smalldot] {};}
\DeclareSymbol{20}{0.7}{\draw (-1,1.8) node[smalldot] {} -- (-0.5,0.9); \draw (0,1.8) node[smalldot] {} -- (-0.5,0.9) -- (-0.5,0);}
\DeclareSymbol{210}{1.1}{\draw (-0.5,2.4) node[smalldot] {} -- (-1,1.6);\draw (-1.5,2.4) node[smalldot] {} -- (-0.5,0.8); \draw (0,1.6) node[smalldot] {} -- (-0.5,0.8) -- (-0.5,0);}
\DeclareSymbol{211}{1.1}{\draw (-0.5,2.4) node[smalldot] {} -- (-1,1.6);\draw (-1.5,2.4) node[smalldot] {} -- (-0.5,0.8); \draw (0,1.6) node[smalldot] {} -- (-0.5,0.8) -- (0,0) -- (0.5,0.8) node[smalldot] {};}
\DeclareSymbol{22}{0.7}{\draw (-1.5,1.8) node[smalldot] {} -- (-1,0.9) -- (0,0) -- (1,0.9) -- (1.5,1.8) node[smalldot] {}; \draw (-0.5,1.8) node[smalldot] {} -- (-1,0.9);\draw (0.5,1.8) node[smalldot] {} -- (1,0.9);}

\setlist[itemize]{topsep=3pt,itemsep=1.5pt,parsep=0pt}

\def\scal#1{\langle#1\rangle}
\def\cent#1{\mathopen{{\langle\kern-0.3em\rangle}}#1\mathclose{{\langle\kern-0.3em\rangle}}}

\def\d{\partial}

\begin{document}

\title{Subcritical approximations to stochastic defocusing mass-critical nonlinear Schr\"odinger equation on $\RR$}
\author{Chenjie Fan$^1$ and Weijun Xu$^2$}
\institute{University of Chicago, US, \email{cjfanpku@gmail.com}
	\and University of Oxford, UK / NYU Shanghai, China, \email{weijunx@gmail.com}}

\maketitle

\begin{abstract}
	We show robustness of various truncated and subcritical approximations to the stochastic defocusing mass-critical nonlinear Schr\"odinger equation (NLS) in dimension $d=1$, whose solution was constructed in \cite{snls_critical} with one particular such approximation. The key ingredient in the proof is a uniform bound of the solutions to the family of deterministic mass-subcritical defocusing NLS. 
\end{abstract}

\setcounter{tocdepth}{2}
\microtypesetup{protrusion=false}
\tableofcontents
\microtypesetup{protrusion=true}

\def\k{\mathbf{k}}

\section{Introduction}
\label{sec:intro}

\subsection{Solution to the mass-critical stochastic equation}

The aim of this article is to show robustness of truncated and subcritical approximations to the mass-critical stochastic nonlinear Schr\"odinger equation
\begin{equation} \label{eq:snls_stratonovich}
i \d_t u + \Delta u = |u|^4 u + u \circ \dot{W}\;, \qquad u(0,\cdot) \in L_{\omega}^{\infty}L_{x}^{2}, 
\end{equation}
where $\dot{W}$ is white in time and coloured in space, and $\circ$ denotes the Stratonovich product that preserves the $L^2$-norm of the solution. The solution to \eqref{eq:snls_stratonovich} was constructed in the recent work \cite{snls_critical} also via approximations of truncated and subcritical problems, but with the limit taken in a particular oder. 

We first give the precise assumption on the noise. Let $\hH$ be the Hilbert space of real-valued functions on $\RR$ with the inner product
\begin{equation*}
\scal{f,g}_{\hH} = \sum_{j=1}^{N} \scal{(1+|x|^K) f^{(j)}, (1+|x|^K) g^{(j)}}_{L^2}
\end{equation*}
for some sufficiently large $K$ and $N$ ($K, N = 10$ would be enough). Our assumption on the noise is the following. 

\begin{assumption} \label{as:noise}
	The Wiener process $W$ has the form $W = \Phi \tilde{W}$, where $\Phi: L^2(\RR) \rightarrow \hH$ is a trace-class operator, and $\tilde{W}$ is the cylindrical Wiener process on $L^2(\RR)$, defined on the probability space $(\Omega, \fF, \PP)$ with natural filtration $(\fF_t)_{t \geq 0}$. The noise $\dot{W}$ in \eqref{eq:snls_stratonovich} is the time derivative of $W$. 
\end{assumption}

With this assumption on the noise, one can re-write \eqref{eq:snls_stratonovich} in its It\^o form as
\begin{equation*}
i \d_t u + \Delta u = |u|^4 u + u \dot{W} - \frac{i}{2} u F_{\Phi}, 
\end{equation*}
where the product between $u$ and $\dot{W}$ is in the It\^o sense, and
\begin{equation*}
F_{\Phi}(x) = \sum_{k} (\Phi e_k)^{2}(x)
\end{equation*}
is the It\^o-Stratonovich correction, which is independent of the choice of orthornormal basis $\{e_k\}$ of $L^2(\RR)$. 

We now introduce a few notations. For every interval $\iI \subset \RR$, let
\begin{equation} \label{eq:notation_space}
\xX_{1}(\iI) = L_{t}^{\infty}L_{x}^{2}(\iI) := L^{\infty}\big(\iI, L^{2}(\RR)\big), \quad \xX_{2}(\iI) = L_{t}^{5}L_{x}^{10}(\iI) := L^{5}\big(\iI, L^{10}(\RR)\big). 
\end{equation}
Let $\xX = \xX_1 \cap \xX_2$ in the sense that
\begin{equation*}
\|\cdot\|_{\xX(\iI)} = \|\cdot\|_{\xX_1(\iI)} + \|\cdot\|_{\xX_2(\iI)}. 
\end{equation*}
Let $(\Omega, \fF, \PP)$ be the probability space as in Assumption~\ref{as:noise}. For every $\rho \geq 1$, we write $L_{\omega}^{\rho} \xX(\iI) = L^{\rho}(\Omega, \xX(\iI))$. The following is the main result of \cite{snls_critical}. 

\begin{thm} [\cite{snls_critical}]
	\label{th:snls_critical}
	Let $W$ be the Wiener process above and $u_0$ be independent of $W$ with $\|u_0\|_{L_{\omega}^{\infty}L_{x}^{2}} < +\infty$. Then, there exists a unique global flow $u$ adapted to the filtration generated by $W$ such that $u \in L_{\omega}^{\rho}\xX(0,T)$ for every $T>0$ and every $\rho \geq 5$, and satisfies
	\begin{equation} \label{eq:duhamel_snls_critical}
	\begin{split}
	u(t) &= e^{it\Delta} u_0 - i \int_{0}^{t} e^{i(t-s)\Delta} \big( |u(s)|^{4} u(s) \big) {\rm d}s\\
	&-i \int_{0}^{t} e^{i(t-s)\Delta} u(s) {\rm d} W_s - \frac{1}{2} \int_{0}^{t} e^{i(t-s)\Delta} \big( F_{\Phi} u(s) \big) {\rm d}s. 
	\end{split}
	\end{equation}
	The equality holds in $L_{\omega}^{\rho}\xX(0,T)$ and the stochastic integral is in the It\^o sense. The solution satisfies the bound
	\begin{equation*}
	\|u\|_{L_{\omega}^{\rho}\xX(0,T)} \leq C = C \big(\rho, T, \|u_0\|_{L_{\omega}^{\infty}L_x^2}\big).
	\end{equation*}
	Furthermore, we have pathwise mass conservation in the sense that $\|u(t)\|_{L_x^2} = \|u_0\|_{L_x^2}$ almost surely for every $t \in [0,T]$. 
\end{thm}

The solution $u \in L_{\omega}^{\rho}\xX(0,T)$ in Theorem~\ref{th:snls_critical} was constructed as follows. We start with the existence of the solution in the truncated subcritical problem. Let $\theta: \RR^+ \rightarrow [0,1]$ be smooth with compact support in $[0,2)$, and $\theta=1$ on $[0,1]$. For every $m>0$, let
\begin{equation*}
\theta_{m}(x) = \theta(x/m). 
\end{equation*}
Also, for every $\eps>0$, we let $\nN^\eps(u) = |u|^{4-\eps} u$ and $\nN(u) = |u|^4 u$. A result in \cite{BD} states that for every $m,\eps>0$ and $u_0 \in L_{\omega}^{\infty}L_{x}^{2}$ independent of $W$, there is a unique process $u_{m,\eps} \in L_{\omega}^{\rho}\xX(0,T)$ satisfying
\begin{equation} \label{eq:duhamel_me}
\begin{split}
u_{m,\eps}(t) = &e^{it\Delta} u_0 - i \int_{0}^{t} e^{i(t-s)\Delta} \Big( \theta_{m}\big( \|u_{m,\eps}\|_{\xX_2(0,s)}^{5} \big) \nN^{\eps} \big( u_{m,\eps}(s) \big) \Big) {\rm d}s\\
&- i \int_{0}^{t} e^{i(t-s)\Delta} u_{m,\eps}(s) {\rm d} W_s - \frac{1}{2} \int_{0}^{t} e^{i(t-s)\Delta} \big( F_{\Phi} u_{m,\eps}(s) \big) {\rm d}s, 
\end{split}
\end{equation}
and one has the pathwise mass conservation $\|u_{m,\eps}(t)\|_{L_{x}^{2}} = \|u_0\|_{L_x^2}$ almost surely. The authors then proved the existence of the limit
\begin{equation} \label{eq:dd_subcritical}
u_{\infty,\eps} := \lim_{m \rightarrow +\infty} u_{m,\eps}
\end{equation}
for every $\eps>0$, and showed that this limit solves the corresponding subcritical equation without the truncation. Note that the above limit relies on the strict positivity of the fixed (though arbitrary) $\eps$. 

On the other hand, in order to construct the solution to the critical equation \eqref{eq:duhamel_snls_critical}, one needs to take \textit{both} limits $m \rightarrow +\infty$ and $\eps \rightarrow 0$. In \cite{snls_critical}, starting from the family $\{u_{m,\eps}\}_{m,\eps}$, we took the following procedures: 
\begin{enumerate}
	\item For every $m>0$, we were able to show that the sequence $\{u_{m,\eps}\}_{\eps}$ converges in $L_{\omega}^{\rho}\xX(0,T)$ to a limit $u_m$, which satisfies
	\begin{equation} \label{eq:duhamel_m}
	\begin{split}
	u_{m}(t) = &e^{it\Delta} u_0 - i \int_{0}^{t} e^{i(t-s)\Delta} \Big( \theta_{m}\big( \|u_{m}\|_{\xX_2(0,s)}^{5} \big) \nN \big( u_{m}(s) \big) \Big) {\rm d}s\\
	&- i \int_{0}^{t} e^{i(t-s)\Delta} u_{m}(s) {\rm d} W_s - \frac{1}{2} \int_{0}^{t} e^{i(t-s)\Delta} \big( F_{\Phi} u_{m}(s) \big) {\rm d}s
	\end{split}
	\end{equation}
	in the same space. 
	
	\item In the second step, starting from the sequence $\{u_m\}_m$ as obtained in the previous step, we were able to show that $u_m \rightarrow u$ in $L_{\omega}^{\rho}\xX(0,T)$ and that the limit $u$ satisfies \eqref{eq:duhamel_snls_critical}. A key ingredient in \cite{snls_critical} is a uniform-in-$m$ bound of $\|u_m\|_{L_{\omega}^{\rho}\xX(0,T)}$. 
\end{enumerate}

In short, we were able to show the existence of the limit
\begin{equation*}
u := \lim_{m \rightarrow +\infty} \lim_{\eps \rightarrow 0} u_{m,\eps}
\end{equation*}
and the corresponding Duhamel's formula \eqref{eq:duhamel_snls_critical} for $u$. The order of the limit taken above was essential in the construction in \cite{snls_critical}, as it relies on the uniform bound on $\{u_m\}$ with $\eps=0$ particularly.

\subsection{Main result and key ingredient}

From the above discussions, it is natural to expect that the solution $u$ can also be approximated by taking $m \rightarrow +\infty$ first and then $\eps \rightarrow 0$. This would require a uniform bound on $\{u_{m,\eps}\}_{m,\eps}$ in both $m$ and $\eps$, but not just in $m$ while $\eps=0$. The following is our main theorem. 

\begin{thm} \label{th:snls_approx}
	Let $u_0 \in L_{\omega}^{\infty}L_{x}^{2}$ be independent of $W$. For every $m>0$ and $\eps \in (0,1)$, let $u_{m,\eps}$ be the solution to \eqref{eq:duhamel_me}. Let $\rho \geq 5$ and $T>0$ be arbitrary. Then there exists $B$ depending on $\rho$, $T$ and $\|u_0\|_{L_{\omega}^{\infty}L_{x}^{2}}$ only such that
	\begin{equation} \label{eq:ume_uniform_bd}
	\|u_{m,\eps}\|_{L_{\omega}^{\rho}\xX(0,T)} \leq B
	\end{equation}
	for all $m>0$ and $\eps \in (0,1)$. Furthermore, let $u$ denote the solution to the critical equation \eqref{eq:duhamel_snls_critical} as in Theorem~\ref{th:snls_critical}, then for every $\delta>0$, there exist $m_0 > 0$ and $\eps_0 \in (0,1)$ such that
	\begin{equation} \label{eq:ume_converge}
	\|u_{m,\eps} - u\|_{L_{\omega}^{\rho}\xX(0,T)} < \delta
	\end{equation}
	for all $m>m_0$ and $\eps \in (0,\eps_0)$. 
\end{thm}

The key ingredient of the proof of Theorem~\ref{th:snls_approx} is the uniform boundedness of the family of solutions $\{w_\eps\}$ to the one-dimensional \textit{deterministic} defocusing Schr\"odinger equation(s)
\begin{equation} \label{eq:nls_eps}
i \d_t w_\eps + \Delta w_\eps = \mu |w_\eps|^{4-\eps} w_\eps\;, \qquad w_\eps(0,\cdot) \in L^{2}(\RR)
\end{equation}
in finite time interval $[0,T]$ over $\eps \in [0,1]$, $\mu \in [0,1]$ and $L^2$-bounded initial data. Once we have such a deterministic uniform bound, the derivation from it to the stochastic bound \eqref{eq:ume_uniform_bd} is essentially the same as the corresponding $\eps=0$ situation in \cite{snls_critical}, and we can establish the convergence \eqref{eq:ume_converge} using the boundedness of $\{u_{m,\eps}\}$ in \eqref{eq:ume_converge}. 

Recall the notations $\xX_1, \xX_2$ and $\xX$ from \eqref{eq:notation_space}. The main deterministic uniform bound is the following. 

\begin{prop} \label{pr:Dodson_eps}
	For every $M>0$ and $T>0$, there exists $B=B(M,T)$ such that if $w_\eps$ solves \eqref{eq:nls_eps} with $\mu \in [0,1]$, $\eps \in [0,1]$ and $\|w_\eps(0)\|_{L_{x}^{2}} \leq M$, then we have
	\begin{equation*}
	\|w_\eps\|_{\xX(0,T)} \leq B. 
	\end{equation*}
	The constant $B$ depends on $M$ and $T$ only. 
\end{prop}

One may wonder whether the above proposition follows immediately from the boundedness of subcritical solution (say $\eps=1$) and Dodson's recent scattering result (\cite{dodson2016global}), which correspond to $\eps=0$. While it is a consequence of these two extreme situations, the proof is somewhat technically involved since Dodon's results rely on the scale invariance of the solution when $\eps=0$, which is not available for positive $\eps$. Instead, we employ concentration compactness together with the boundedness of the two extreme cases to establish the uniform boundedness of $\{w_\eps\}$. 

Note that we need the bound in Proposition~\ref{pr:Dodson_eps} to be uniform both in $\eps \in [0,1]$ and in the class of initial data with $L^2$-norm bounded by $M$. If we fix the $L^2$ initial data and only requires uniformity in $\eps$, then it indeed follows directly from Dodson's theorem (when $\eps=0$) and a standard persistence of regularity argument.

\subsection{Structure of the article}

The article is organised as follows. In Section~\ref{sec:pre}, we briefly review the background on the deterministic nonlinear Schr\"odinger equation, including dispersive and Strichartz estimates, local well-posedness and Dodson's results on the global bounds and stability for mass-critical NLS. Sections~\ref{sec:road_eps} and~\ref{sec:main_tech} are devoted to the proof of Proposition~\ref{pr:Dodson_eps}. In Section~\ref{sec:road_eps}, we introduce concentration compactness, which is the main tool of the proof, and then re-formulate Proposition~\ref{pr:Dodson_eps} to Proposition~\ref{pr:main_tech}. We then give a complete proof of Proposition~\ref{pr:main_tech} in Section~\ref{sec:main_tech} with concentration compactness. Finally, in Section~\ref{sec:stochastic}, we prove Theorem~\ref{th:snls_approx} using Proposition~\ref{pr:Dodson_eps}. This also relies on a series of other uniform boundedness/stability statements for deterministic equations.

\subsection{Notations}

We recall the notations introduced above. For every interval $\iI$, we let
\begin{equation*}
\xX_1(\iI) = L_{t}^{\infty}L_{x}^{2}(\iI) = L^{\infty}\big(\iI, L^{2}(\RR)\big)\;, \quad \xX_2(\iI) = L_{t}^{5}L_{x}^{10}(\iI) = L^{5}\big(\iI, L^{10}(\RR)\big),  
\end{equation*}
and define $\xX(\iI)$ such that $\|\cdot\|_{\xX(\iI)} = \|\cdot\|_{\xX_1(\iI)} + \|\cdot\|_{\xX_2(\iI)}$. We also write $L_{\omega}^{\rho}\yY = L^{\rho}(\Omega, \yY)$. For $\eps \in [0,1]$, let $\nN^\eps(u) = |u|^{4-\eps} u$, and $\nN(u) = |u|^{4}u$. We let $\theta$ denote smooth non-negative function defined on $[0+\infty)$ such that it has compact support in $[0,2]$, and equals $1$ on $[0,1]$. Finally, for $m>0$, we let $\theta_{m}(\cdot) = \theta(\cdot/m)$. 

Also, in most places in the article (except Sections~\ref{sec:road_eps} and~\ref{sec:main_tech}), we use $u$ (as well as with suitable subscripts) to denote the solution to the stochastic equation, $w$ (or $w_\eps$, $w_{m,\eps}$) to denote the solution to the deterministic mass-critical (or subcritical) equation, and $v$ (also with $v_{\eps}$ and $v_{m,\eps}$) to denote the perturbed versions of $w$ (and $w_\eps$, $w_{m,\eps}$). On the other hand, Sections~\ref{sec:road_eps} and ~\ref{sec:main_tech} are technical and various equations/solutions are involved, so the notations there (such as $u$, $v$, $w$) are independent with the rest of the article.

\subsection*{Acknowledgement}
{\small We thank Carlos Kenig for helpful discussions. WX acknowledges the support from the Engineering and Physical Sciences Research Council through the fellowship EP/N021568/1.

The content of this article is contained in the authors' preprint arXiv:1807.04402. The other part of arXiv:1807.04402, which has not been covered by this article, has been combined with arXiv:1803.03257 to form the article \cite{snls_critical}. Only \cite{snls_critical} and the current article will be submitted for journal publication. 
}

\section{Preliminaries on Strichartz estimates and deterministic NLS}
\label{sec:pre}

\subsection{Dispersive and Strichartz estimtes}

We state some dispersive and Strichartz estimates below which are fundamental in the study of Schr\"odinger equation and are used throughout the article. These estimates are all stated in space dimension $d=1$. One may refer to \cite{cazenave2003semilinear} \cite{keel1998endpoint},\cite{tao2006nonlinear}  and reference therein.

\begin{prop} \label{pr:dispersive}
	There exists $C>0$ such that
	\begin{equation*}
	\|e^{it\Delta} f\|_{L^{p'}(\RR)} \leq C t^{\frac{1}{p}-\frac{1}{2}} \|f\|_{L^{p}(\RR)}
	\end{equation*}
	for every $p \in [1,2]$ and every $f \in L^{p}(\RR)$. Here, $p'$ is the conjugate of $p$. 
\end{prop}

We now state Strichartz estimates. A pair of non-negative real numbers $(q,r)$ is called an admissible pair (for $d=1$) if
\begin{equation*}
\frac{2}{q} + \frac{1}{r} = \frac{1}{2}. 
\end{equation*}
We have the following Strichartz estimates. 

\begin{prop}
	For every admissible pair $(q,r)$, there exists $C>0$ depending on $(q,r)$ only such that 
	\begin{equation*}
	\|e^{it\Delta} f\|_{L_{t}^{q}L_{x}^{r}(\RR)} \leq C \|f\|_{L_{x}^{2}}
	\end{equation*}
	for all $f \in L_{x}^{2}$. For every two admissible pairs $(q,r)$ and $(\tilde{q}, \tilde{r})$, there exists $C>0$ such that
	\begin{equation*}
	\Big\|\int_{\tau}^{t} e^{i(t-s)\Delta} f(s) {\rm d}s\Big\|_{L_{t}^{q}L_{x}^{r}(\iI)} \leq \|f\|_{L_{t}^{\tilde{q}'}L_{x}^{\tilde{r}'}(\iI)}
	\end{equation*}
	for all interval $\iI \subset \RR$ and all space-time functions $f \in L_{t}^{\tilde{q}'}L_{x}^{\tilde{r}'}(\iI)$. Here, $\tilde{q}'$ and $\tilde{r}'$ are the conjugates of $q$ and $r$. 
\end{prop}

\subsection{Preliminary for local theory}

It is now standard to prove local well posedness for mass critical/subcrtical NLS. We briefly review it for the convenience of the readers. One may refer to \cite{cazenave1989some}, \cite{cazenave2003semilinear} and \cite{tao2006nonlinear} for more details.
We present the following lemmas to summarise the key estimates in local well-posedness. They will be used in various places in Section~\ref{sec:main_tech}

Since we will prove Theorems~\ref{th:Dodson} and~\ref{th:snls_approx} for $T=1$, we give these statements on sub-intervals $\iI=[a,b] \subset [0,1]$. Recall that $w_\eps$ satisfies the equation
\begin{equation*}
i \d_t w_\eps + \Delta w_\eps = \mu \nN^\eps(w_\eps)\;, \qquad w_\eps(a,\cdot) \in L_x^2, 
\end{equation*}
where $\eps, \mu \in [0,1]$. All the bounds below are uniform over $\eps$, $\mu$ and the interval $\iI \subset [0,1]$. 

\begin{lem} \label{le:pre_lwp}
	There exists $C>0$ such that
	\begin{equation*}
	\|w_{\eps} - e^{ i t \Delta}w_{\eps}(a)\|_{\xX_2(\iI)} \leq C (b-a)^{\frac{\eps}{4}} \|w_\eps\|_{\xX_{1}(\iI)} \|u_\eps\|_{\xX_{2}(\iI)}^{4-\eps}. 
	\end{equation*}
\end{lem}

The following lemma gives the bound on $\|w_\eps\|_{\xX_2(\iI)}$ when $\|w_\eps(a)\|_{L_x^2}$ is small. 

\begin{lem} \label{le:small_mass}
	There exists $ \delta_{0}>0$ and $C>0$ such thats if $\|w_\eps(0)\|_{L_x^2} \leq \delta_0$, then
	\begin{equation*}
	\sup_{\eps \in [0,1]} \|u_{\eps}\|_{\xX_2(\iI)} \leq C. 
	\end{equation*}
\end{lem}

In practice, we need to slightly enhance Lemma \ref{le:small_mass} to the following. 

\begin{lem} \label{le:small_enhance}
	Suppose $\|w_\eps(a)\|_{L_x^2} \leq M$. There exists $\delta_0>0$ depending on $M$ and $C>0$ universal such that if
	\begin{equation*}
	\|e^{i t \Delta} w_{\eps}(a)\|_{\xX_2(\iI)} \leq \delta_{0}, 
	\end{equation*}
	then we have
	\begin{equation*}
	\|w_{\eps}\|_{\xX_2(\iI)} \leq C M \|e^{i t \Delta} w_{\eps}(a)\|_{\xX_2(\iI)}. 
	\end{equation*}
\end{lem}

The next proposition is a conditional stability statement. 

\begin{prop} \label{pr:iteam_eps}
	Let $w_\eps$ be as above, and $v_\eps \in \xX(\iI)$, $e \in L_{t}^{1}L_{x}^{2}(\iI)$ such that
	\begin{equation*}
	i \d_t v_\eps + \Delta v_\eps = \mu \nN^\eps(v_\eps) + e\;, \qquad v_\eps(a,\cdot) \in L_x^2. 
	\end{equation*}
	Then for every $M_1, M_2>0$, there exist $\delta, C>0$ such that if
	\begin{equation*}
	\|v_\eps\|_{\xX_1(\iI)} \leq M_1\;, \quad \|v_\eps\|_{\xX_2(\iI)} \leq M_2\;, \quad \|v_\eps(a)-w_\eps(a)\|_{L_x^2} + \|e\|_{L_{t}^{1}L_{x}^{2}(\iI)} \leq \delta\;, 
	\end{equation*}
	we have
	\begin{equation*}
	\|v_\eps - w_\eps\|_{\xX(\iI)} \leq C \Big( \|v_\eps(a)-w_\eps(a)\|_{L_x^2} + \|e\|_{L_{t}^{1}L_{x}^{2}(\iI)} \Big). 
	\end{equation*}
	The constants $\delta$ and $C$ depend on $M_1$ and $M_2$ only. 
\end{prop}

This is an uniform-in-$\eps$ version of \cite[Lemmas~3.9 and~3.10]{colliander2008global}. The proof is essentially the same, so we omit the it here. Note that this is a perturbative statement since one assumes a bound on $\|v_\eps\|_{\xX_2(\iI)}$. The statements given in the next subsection however, are totally non-perturbative.

\subsection{Boundedness and stability of mass-critical NLS}

We briefly review the boundedness and stability results for the mass-critical NLS, highlighting Dodson's theorem. These have been reviewed in \cite[Section~4.1]{snls_critical}, but since these statements will be used later, we still summarise them here for completeness. Throughout, let $\iI = [a,b]$ with $b-a \leq 1$, and fix $\mu \in [0,1]$. All the bounds below are uniform in $\mu \in [0,1]$, and their dependence on $\iI$ is through $b-a$ only. 

Let $w \in \xX(\iI)$ be the solution to the mass-critical NLS
\begin{equation} \label{eq:snls_w}
i \d_t w + \Delta w = \mu |w|^4 w\;, \qquad w(a) \in L_{x}^{2}(\RR). 
\end{equation}
Let $v \in \xX(\iI)$ and $e \in L_{t}^{1}L_{x}^{2}(\iI)$ such that
\begin{equation} \label{eq:snls_v}
i \d_t v + \Delta v = \mu |v|^4 v + e\;, \qquad v(a) \in L_{x}^{2}(\RR). 
\end{equation}
We have the following statements. 

\begin{thm} [Dodson]
	\label{th:Dodson}
	For every $w(a) \in L_{x}^{2}$, the equation \eqref{eq:snls_w} has a global solution $w \in \xX(\RR)$. Moreover, for every $M>0$, there exists $D_M>0$ such that
	\begin{equation*}
	\|w\|_{\xX(\RR)} \leq D_{M}
	\end{equation*}
	whenever $w(a) \leq M$. 
\end{thm}

With the bound in Theorem~\ref{th:Dodson}, one can enhance Proposition~\ref{pr:iteam_eps} to the following by removing the assumption on the $\xX_2$-norm bound of the solution. 

\begin{prop} \label{pr:strong_sta}
	Let $w$ be the solution to \eqref{eq:snls_w}, and $v \in \xX(\iI)$ and $e \in L_{t}^{1}L_{x}^{2}(\iI)$ satisfy \eqref{eq:snls_v}. For every $M>0$, there exist $\delta_M, C_M>0$ such that if
	\begin{equation*}
	\|w(a)\|_{L_{x}^{2}} \leq M \qquad \text{and} \qquad \|w(a)-v(a)\|_{L_x^2} + \|e\|_{L_{t}^{1}L_{x}^{2}} \leq \delta_M, 
	\end{equation*}
	then we have
	\begin{equation*}
	\|w-v\|_{\xX(\iI)} \leq C_{M} \Big( \|w(a)-v(a)\|_{L_x^2} + \|e\|_{L_{t}^{1}L_{x}^{2}} \Big). 
	\end{equation*}
	The constants $\delta_M$ and $C_M$ depend on $M$ only, and are uniform over $|\iI| \leq 1$ and $\mu \in [0,1]$. 
\end{prop}

\section{Overview for the proof of Proposition~\ref{pr:Dodson_eps}} 
\label{sec:road_eps}

By pathwise mass conservation, it suffices to prove the theorem for $T=1$. So in the rest of this and the next section, we always set $\iI=[0,1]$. The boundedness of $w_\eps$ in Proposition~\ref{pr:Dodson_eps} when $\eps$ is away from $0$ follows from standard local theory for subcritical NLS. On the other hand, the extreme case $\eps=0$ is the celebrated result of Dodson (Theorem~\ref{th:Dodson}). It then suffices to obtain a uniform-in-$\eps$ bound as $\eps \rightarrow 0$. 

At first glance, it is natural to consider a direct argument via persistence of regularity. However, when one slightly perturb the $L_x^2$ initial data to a smooth function (with difference in $L_x^2$-norm smaller than $\delta$), the smooth norm of the perturbed data will in general depend on $\delta$ and the actual initial data (not only its $L_x^2$ norm). Hence, such arguments will only give uniform-in-$\eps$ boundedness for every fixed $L^2$ initial data, but not uniform over initial data with bounded $L^2$ norm. 

The main idea is to obtain uniformity not only in $\eps$ but also $L^2$-bounded initial data is to use concentration compactness to localize the initial data. Concentration compactness is also often referred as profile decomposition in the literature. We start with introducing the notions.

\subsection{Concentration compactness}

\begin{defn}
	For every $x_0, \xi_0 \in \RR^d$, $\lambda_0 \in \RR^+$ and $t_0 \in \RR$, we define a unitary operation $\fg$ on $L_{x}^{2}(\RR^d)$ by
	\begin{equation} \label{eq:group_action}
	(\fg_{x_{0},\xi_{0},\lambda_{0},t_{0}}f)(x) := \lambda_0^{-\frac{d}{2}} e^{ix \cdot \xi_{0}} \big(e^{i(-\frac{t_{0}}{\lambda_{0}^{2}})\Delta}f \big)\big(\frac{x-x_{0}}{\lambda_{0}}\big).
	\end{equation}
	The collection $\fG:=\{\fg_{x_{0},\xi_{0},\lambda_{0}, t_{0}}| x_{0},\xi_{0}\in \RR^{d}, \lambda_{0}\in \RR^{+}, t_{0}\in \RR\}$ is then a unitary group acting on $L_{x}^{2}(\RR^d)$. 
\end{defn}

Now we are ready to state the concentration compactness. 

\begin{prop}[\cite{merle1998compactness, carles2007role, begout2007mass, tao2008minimal}] 
	\label{pr:profile_decomp}
	Let $(q,r)$ be an admissible pair in dimension $d$ in the sense that
	\begin{equation*}
	\frac{2}{q} + \frac{d}{r} = \frac{d}{2}. 
	\end{equation*}
	Let $\{f_{n}(x)\}_{n=1}^{\infty}$ be a bounded sequence in $L_{x}^{2}(\RR^{d})$. Up to picking a subsequence, there exist a family of $L^{2}(\RR^{d})$ functions $\{\phi_{j}\}_{j \geq 1}$ and parameters $\{x_{j,n}, \xi_{j,n}, \lambda_{j,n}, t_{j,n}\}_{j,n}$ such that with the shorthand notation $\fg_{j,n} = \fg_{x_{j,n}, \xi_{j,n}, \lambda_{j,n}, t_{j,n}}$, the decompositions
	\begin{equation}\label{eq:decom}
	f_{n}= \sum_{j=1}^{J} \fg_{j,n}\phi_{j} + \omega_{n}^{J},
	\end{equation}
	satisfy the following properties: 
	\begin{itemize}
		\item For every $j \neq j'$, we have
		\begin{equation} \label{eq:separate}
		\lim_{n\rightarrow \infty} \Big( \frac{\lambda_{j,n}}{\lambda_{j',n}}+\frac{\lambda_{j',n}}{\lambda_{j,n}}+ \big|\frac{x_{j,n}-x_{j',n}}{\lambda_{j,n}}\big| + \big|\lambda_{j,n}(\xi_{j,n}-\xi_{j',n})\big| + \big|\frac{t_{j,n}-t_{j',n}}{\lambda_{j,n}^{2}}\big|\Big) = +\infty
		\end{equation}
		
		\item For every $J \geq 1$, we have
		\begin{equation}\label{eq:ortho_mass}
		\lim_{n \rightarrow +\infty} \Big|\|f_{n}\|_{L_{x}^{2}}^{2}-\sum_{j\leq J}\|\phi_{j}\|_{L_{x}^{2}}-\|\omega^{J}_{n}\|_{L_{x}^{2}}^{2} \Big|=0.
		\end{equation}
		
		\item The remainder term $\omega_{n}^{J}$ satisfies
		\begin{equation} \label{eq:small_remainder}
		\lim_{J \rightarrow\infty} \limsup_{n \rightarrow +\infty} \|e^{it\Delta}\omega_{n}^{J}\|_{L_{t}^{q}L_{x}^{r}(\RR \times \RR^d)}=0.
		\end{equation}

		\item Finally, for every $J \geq 1$ and every $j \geq J$, we have $\fg_{j,n}^{-1} \omega^{J}_{n}\rightharpoonup 0$ weakly in $L^2$. 
	\end{itemize}
	We call each $\phi_j$ a profile with parameters $\{x_{j,n}, \xi_{j,n}, \lambda_{j,n}, t_{j,n}\}_{n}$, and also say  $\{f_{n}\}_{n}$ admits a profile decomposition with profiles $\{\phi_{j}; \{x_{j,n}, \xi_{j,n}, \lambda_{j,n},t_{j,n}\}_{n}\}_{j}$. 
\end{prop}

\begin{rem} \label{rem:symmetry}
	The above parameters take into account the symmetries for both mass critical nonlinear Schr\"odinger equation and linear Schr\"odinger equation. In fact, if $\{x_0, \xi_0, \lambda_0, t_0\}$ is a parameter set, and $\Psi$ solves
	\begin{equation*}
	i \d_t \Psi + \Delta \Psi = \mu |\Psi|^{\frac{4}{d}} \Psi\;, \quad \Psi(-t_{0}/\lambda_{0}^{2}\;, \cdot) = \phi \in L_{x}^{2}(\RR^d), 
	\end{equation*}
	then the space-time function
	\begin{equation*}
	\Phi(t,x) := \lambda_{0}^{-\frac{d}{2}} e^{i x \cdot \xi_0} e^{-it |\xi_0|^{2}} \Psi\Big( \frac{t-t_0}{\lambda_{0}^{2}}, \frac{x - x_{0} - 2 \xi_0 t}{\lambda_0} \Big)
	\end{equation*}
	satisfies the same equation with initial data $\Phi(0,x) = \lambda_{0}^{-\frac{d}{2}} e^{i x \cdot \xi_0} \phi\big(\frac{x-x_0}{\lambda_0}\big)$. Furthermore, it preserves the Strichartz norm in the sense that
	\begin{equation*}
	\|\Psi\|_{L_{t}^{q}L_{x}^{r}(\RR \times \RR^d)} = \|\Phi\|_{L_{t}^{q}L_{x}^{r}(\RR \times \RR^d)}
	\end{equation*}
	for every admissible pair $(q,r)$. 
\end{rem}

\begin{rem} \label{rem:orthogonality}
	The separation property \eqref{eq:separate} has important consequences. For example, for $(q,r)$ admissible, we have
	\begin{equation*}
	\lim_{n \rightarrow +\infty} \|e^{it\Delta}(\fg_{j,n} \phi_j) e^{it\Delta}(\fg_{j',n} \phi_{j'})\|_{L_{t}^{\frac{q}{2}}L_{x}^{\frac{r}{2}}(\RR \times \RR^d)} = 0
	\end{equation*}
	whenever $j \neq j'$. Note that this does not indicate any sense of orthogonality of $\phi_j$ and $\phi_{j'}$! But rather it is a direct consequence of the property \eqref{eq:separate}. In fact, more generally, for $(q,r)$ admissible and $\Psi, \tilde{\Psi}$ in $L_{t}^{q}L_{x}^{r}$, let
	\begin{equation*}
	\begin{split}
	\Psi_{j,n}(t,x) &= \lambda_{j,n}^{-\frac{d}{2}} e^{ix \cdot \xi_{j,n}}e^{-it |\xi_{j,n}|^{2}} \Psi \Big(\frac{t-t_{jn}}{\lambda_{j,n}^{2}},\frac{x-x_{j,n}-2\xi_{j,n}t}{\lambda_{j,n}} \Big)\;, 
	\end{split}
	\end{equation*}
	and define $\tilde{\Psi}_{j',n}$ in a similar way. Then, \eqref{eq:separate} implies
	\begin{equation} \label{eq:ortho_stri}
	\lim_{n \rightarrow +\infty} \|\Psi_{j,n} \tilde{\Psi}_{j',n}\|_{L_{t}^{q/2}L_{x}^{r/2}(\RR \times \RR^{d})}=0
	\end{equation}
	for every $j \neq j'$. 
\end{rem}

\begin{rem} \label{rem:limits_exist}
	A typical feature in the application of concentration compactness is that a lot of subsequences will be taken. Hence, up to taking subsequences, we assume without loss of generality that for every $j$, all the following limits (as $n \rightarrow +\infty$) exist (including $\pm \infty$): 
	\begin{equation*}
	-\frac{t_{j,n}}{\lambda_{j,n}^{2}}\;, \qquad t_{j,n}\;, \qquad \lambda_{j,n}\;, \qquad \lambda_{j,n}^{\eps_n}\;, 
	\end{equation*}
	where $\{\eps_n\}$ is the sequence in $(0,1)$ taken in Proposition~\ref{pr:main_tech} below. Furthermore, 
	\begin{enumerate}
		\item if $\lim_{n}(-t_{j,n}/\lambda_{j,n}^{2})$ is finite, then we can assume $t_{j,n} \equiv 0$ for all $n$.  . 
		
		\item if $\lim_{n} \lambda_{j,n}>0$ and is finite, we can assume $\lambda_{j,n} \equiv 1$ for all $n$ . 
	\end{enumerate}
	This can be achieved by further adjusting the profile without breaking the properties in Proposition~\ref{pr:profile_decomp}. 
\end{rem}

With Remark~\ref{rem:limits_exist}, we have the following standard notion. 

\begin{defn} \label{de:profile}
	We call a profile $\phi$ with parameters $\{x_{n},\xi_{n},\lambda_{n},t_{n}\}_{n}$
	\begin{enumerate}
		\item Compact profile if $t_{n}\equiv 0$,
		\item Backward scattering profile if $\lim_{n\rightarrow \infty} (-\frac{t_{n}}{\lambda_{n}^{2}} )=-\infty$.
		\item Forward scattering profile if $\lim_{n\rightarrow \infty}(-\frac{t_{n}}{\lambda_{n}^{2}})=+\infty$. 
	\end{enumerate}
\end{defn}

\subsection{Another uniform boundedness proposition}

We now give the key uniform boundedness statement that will imply Proposition~\ref{pr:Dodson_eps}. The form of its statement is more suitable for the use of concentration compactness. 

\begin{prop} \label{pr:main_tech}
	Let $\{\eps_n\} \in (0,1)$ with $\eps_n \rightarrow 0$. Let $\{f_n\}$ be a bounded sequence in $L_x^2$ with $\|f_n\|_{L_x^2} \leq M_0$ and that it admits the profile decomposition as in Proposition~\ref{pr:profile_decomp} and Remarks~\ref{rem:limits_exist}. Let $u_n$ be the solution to the equation
	\begin{equation*}
	i \d_t u_n + \Delta u_n = |u_n|^{4-\eps_n} u_n\;, \qquad u_n(0,\cdot) = f_n. 
	\end{equation*}
	Then, there exists $C > 0$ depending on $M_0$ only such that
	\begin{equation*}
	\limsup_{n \rightarrow +\infty} \|u_n\|_{\xX(\iI)} \leq C. 
	\end{equation*}
\end{prop}

\begin{rmk}
	Note that Proposition~\ref{pr:main_tech} is not totally perturbative since we will use Dodson's result (Theorem~\ref{th:Dodson}) in its proof. 
\end{rmk}

\subsection{Proposition~\ref{pr:main_tech} implies Proposition~\ref{pr:Dodson_eps}}

\begin{proof} [Proof of Proposition~\ref{pr:Dodson_eps} assuming Proposition~\ref{pr:main_tech}]
	Suppose Proposition~\ref{pr:Dodson_eps} is not true. Then, there exists $M_0>0$, a sequence $\{\eps_n\} \in (0,1)$ with $\eps_n \rightarrow 0$ and $\{f_n\} \in L_{x}^{2}$ with $\|f_n\|_{L_{x}^{2}} \leq M_0$ such that the solution $u_n$ to the equation
	\begin{equation*}
	i \d_t u_n + \Delta u_n = |u_n|^{4-\eps_n} u_n\;, \qquad u_n(0, \cdot) = f_n
	\end{equation*}
	diverges in the sense that $\|u_n\|_{\xX_2(\iI)} \geq n$ for every $n$. On the other hand, by Proposition~\ref{pr:profile_decomp}, there exists a subsequence of $\{f_n\}$, which we still denote by $\{f_n\}$, that admits a profile decomposition. Proposition~\ref{pr:main_tech} then implies that $\sup_n \|u_n\|_{\xX_2(\iI)} < +\infty$, a contradiction. Hence, we conclude that Proposition~\ref{pr:main_tech} implies Proposition~\ref{pr:Dodson_eps}. 
\end{proof}

\section{Proof of Proposition~\ref{pr:main_tech}} \label{sec:main_tech}

It remains to prove Proposition~\ref{pr:main_tech}. By assumption, $\{f_n\}$ admits a profile decomposition $\{\phi_{j}, \{x_{j,n}, \xi_{j,n}, \lambda_{j,n}, t_{j,n}\}_{n}\}_{j}$. Recall the definition of $\fg$ from \eqref{eq:group_action} and the short hand notation $\fg_{j,n}$ for $g_{x_{j,n}, \xi_{j,n}, \lambda_{j,n}, t_{j,n}}$. Write $\phi_{j,n} = \fg_{j,n} \phi_j$. Let $v_{j,n}$ be the solution to
\begin{equation*}
i \d_t v_{j,n} + \Delta v_{j,n} = \nN^{\eps_n}(v_{j,n})\;, \qquad v_{j,n}(0, \cdot) = \phi_{j,n}. 
\end{equation*}
The idea is to use $V_{J,n} = \sum_{j=1}^{J} v_{j,n} + \omega_{n}^{J}$ as an approximation to $u_n$ for large $n$ and $J$. To establish uniform boundedness of $V_{J,n}$, we compare each $v_{j,n}$ with its corresponding $\Psi_{j,n}$, which either solves the linear Schr\"odinger equation or the mass-critical nonlinear one, depending on details of the profile $\phi_j$. The desired bound for $V_{J,n}$ and $u_n$ then follows from the known results of $\Psi_{j,n}$.

\subsection{The key comparison statement}

We first specify the comparison building blocks $\Psi_{j,n}$. For each $j$, let $\Psi_j$ be the solution to the equation
\begin{equation} \label{eq:Psi_j}
i \d_t \Psi_{j} + \Delta \Psi_j = \mu_{j} |\Psi_j|^{4} \Psi_j
\end{equation}
such that
\begin{equation} \label{eq:Psi_j_constraint}
\lim_{n \rightarrow +\infty} \|\Psi_{j}(-t_{j,n}/\lambda_{j,n}^{2}, \cdot) - e^{i(-t_{j,n}/\lambda_{j,n}^{2})\Delta}\phi_j\|_{L_{x}^{2}} = 0, 
\end{equation}
where $\mu_j \in [0,1]$ depends on the profile $(\phi_j, \{\lambda_{j,n}, t_{j,n}\}_{n})$ explicitly. Given such $\Psi_j$, we define $\Psi_{j,n}$ by
\begin{equation*}
\Psi_{j,n}(t,x) := \lambda_{j,n}^{-\frac{1}{2}} e^{ix \xi_{j,n}} e^{-it|\xi_{j,n}|^{2}} \Psi_{j} \Big( \frac{t-t_{j,n}}{\lambda_{j,n}^{2}}, \frac{x-x_{j,n}-2 \xi_{j,n} t}{\lambda_{j,n}} \Big). 
\end{equation*}
Now we specify the dependence of $\mu_j$ on the profile $\phi_j$. 
\begin{enumerate}
	\item If $\lambda_{j,n} \rightarrow +\infty$, then $\mu_j = 0$. 
	
	\item If $\lambda_{j,n}=1$ for all $n$, then
	\begin{enumerate} [label=(\alph*)]
		\item $\mu_j = 0$ if $\phi_j$ is a forward scattering or backward scattering profile. 
		
		\item $\mu_j=1$ if $\phi_j$ is a compact profile. 
	\end{enumerate}
	
	\item If $\lambda_{j,n} \rightarrow 0$, then
	\begin{enumerate} [label=(\alph*)]
		\item $\mu_j = 0$ if $\phi_j$ is a forward scattering profile. 
		
		\item $\mu_j = \lim_{n} \sqrt{\lambda_{j,n}^{\eps_n}}$ if $\phi_j$ is a compact profile. 
		
		\item We further divide into two sub-cases if $\phi_j$ is backward scattering: 
		\begin{itemize}
			\item if $\lim_{n} t_{j,n} > 1$, then $\mu_j=0$. 
			\item if $\lim_{n} t_{j,n} \leq 1$, then $\mu_j = \lim_{n} \sqrt{\lambda_{j,n}^{\eps_n}}$. 
		\end{itemize}
	\end{enumerate}
\end{enumerate}

\begin{rmk}
	According to Remark~\ref{rem:limits_exist} and Definition~\ref{de:profile}, the above situations have included all possibilities. The case $\mu_j=0$ precisely corresponds to the linear solution $\Psi_j = e^{it\Delta} \phi_j$. Although we write it as a degenerate case of \eqref{eq:Psi_j}, the behavior of $\Psi_{j,n}$ when $\mu_j=0$ is very different from that of $\mu_j > 0$, as can be seen from the proof below. 
\end{rmk}

The main technical lemma is the following. 

\begin{lem} \label{le:one_pro_approx}
	Let $u_{n}, f_{n},\phi_{j}, \phi_{j,n} ,\Psi_{j,n}, \eps_{n}$ be as in Proposition \ref{pr:main_tech}. Let $v_{j,n}$ solve the equation
	\begin{equation}\label{eq:v_jn}
	i\d_{t}v_{j,n}+\Delta v_{j,n} = \nN^{\eps_{n}}(v_{j,n})\;, \quad v_{j,n}(0)=\phi_{j,n}.
	\end{equation}
	Then, we have
	\begin{equation}\label{eq:asym_approx}
	\lim_{n\rightarrow \infty}\|v_{j,n}-\Psi_{j,n}\|_{\xX(\iI)}=0
	\end{equation}
	for every $j \geq 1$. 
\end{lem}

\subsection{Proof of Lemma~\ref{le:one_pro_approx}}

We first make a few simplifications. Note that if $w$ solves
\begin{equation*}
i \d_t w + \Delta w = \mu |w|^{4-\eps} w, 
\end{equation*}
then $u(t,x) = e^{i x \cdot \xi_0} e^{-it |\xi_0|^{2}} w(t, x-x_0 - 2 \xi_0 t)$ solves the same equation with initial data
\begin{equation*}
u(0,x) = e^{i x \cdot \xi_0} w(0,x-x_0). 
\end{equation*}
The transform does not depend on $\eps$. As a consequence, if we let $\tilde{v}_{j,n}$ and $\tilde{\Psi}_{j,n}$ be the same as $v_{j,n}$ and $\Psi_{j,n}$ except that $x_{j,n}$ and $\xi_{j,n}$ are replaced by some other $\tilde{x}_{j,n}$ and $\tilde{\xi}_{j,n}$, then we have
\begin{equation*}
\|\tilde{v}_{j,n}(t) - \tilde{\Psi}_{j,n}(t)\|_{L_{x}^{p}} = \|v_{j,n}(t)-\Psi_{j,n}(t)\|_{L_{x}^{p}}
\end{equation*}
for all $p \geq 1$ and all $t \in \RR$. Hence, we can assume without loss of generality that $x_{j,n} = \xi_{j,n} = 0$ for all $j$ and $n$, and still using the notation $v$ and $\Psi$ instead of $\tilde{v}$ and $\tilde{\Psi}$. 

Also, for notational convenience, we fix $j \geq 1$, and write $\phi, \Psi$ instead of $\phi_j$ and $\Psi_j$. The parameters associated with $\phi$ is $\{\lambda_n, t_n\}_{n}$ (since we assume $x_n = \xi_n = 0$). We also write $\phi_{n}$ instead of $\phi_{j,n}$, so
\begin{equation*}
\phi_n(x) = \lambda_{n}^{-\frac{1}{2}} (e^{i(-t_n/\lambda_n^2)\Delta} \phi)(x). 
\end{equation*}
Similarly, we write $\mu, \Psi, \Psi_n$ and $v_n$ instead of $\mu_j, \Psi_j, \Psi_{j,n}$ and $v_{j,n}$. Here, $v_n$ solves the equation
\begin{equation*}
i \d_t v_n + \Delta v_n = \nN^{\eps_n}(v_n)\;, \quad v_n(0, \cdot) = \phi_n, 
\end{equation*}
and
\begin{equation*}
\Psi_n(t,x) = \lambda_{n}^{-\frac{1}{2}} \Psi\Big(\frac{t-t_n}{\lambda_n^2}, \frac{x}{\lambda_n}\Big). 
\end{equation*}
Our aims is to show that
\begin{equation*}
\lim_{n \rightarrow +\infty} \|v_n - \Psi_n\|_{\xX(\iI)}=0. 
\end{equation*}
We now start to treat the possible situations one by one.

\subsubsection{Case 1: $\lambda_n \rightarrow +\infty$}

In this case, $\Psi = e^{it\Delta} \phi$ is the solution to the free equation. By the definitions of $\Psi_n$ and $\phi_n$, we also have $\Psi_n = e^{it\Delta} \phi_n$. Since $\|e^{it\Delta} \phi\|_{\xX_2(\RR)} < +\infty$, the assumption $\lambda_n \rightarrow +\infty$ implies that
\begin{equation*}
\|e^{it\Delta} \phi_n\|_{\xX_2(\iI)} = \|e^{it\Delta} \phi\|_{\xX_2(-\frac{t_n}{\lambda_n^2}, \frac{1-t_n}{\lambda_n^2})} \rightarrow 0
\end{equation*}
since the length of the time interval shrinks to $0$. Hence, by Lemma~\ref{le:small_enhance}, we also have
\begin{equation*}
\lim_{n \rightarrow +\infty} \|v_n\|_{\xX_2(\iI)} \leq C \lim_{n \rightarrow +\infty} \|e^{it\Delta} \phi_n\|_{\xX_2(\iI)} = 0. 
\end{equation*}
This clearly implies $\|v_n - \Psi_n\|_{\xX_2(\iI)} \rightarrow 0$ if $\lambda_n \rightarrow +\infty$.

\subsubsection{Case 2: $\lambda_n \equiv 1$}

If $\phi$ is forward scattering, then $-\frac{t_n}{\lambda_n^2} \rightarrow +\infty$. If $\phi$ is backward scattering, then $\lambda_n=1$ implies $\frac{1-t_n}{\lambda_n^2} \rightarrow -\infty$. In both situations, we have
\begin{equation*}
\|e^{it\Delta} \phi_n\|_{\xX_2(\iI)} = \|e^{it\Delta} \phi\|_{\xX_2(-\frac{t_n}{\lambda_n^2}, \frac{1-t_n}{\lambda_n^2})} \rightarrow 0. 
\end{equation*}
Again, Lemma~\ref{le:small_enhance} implies that for $\Psi = e^{it\Delta} \phi$, we have
\begin{equation*}
\lim_{n \rightarrow +\infty} \|v_n\|_{\xX_2(\iI)} \leq C \lim_{n \rightarrow +\infty} \|e^{it\Delta} \phi_n\|_{\xX_2(\iI)} = 0, 
\end{equation*}
which gives the desired claim. 

We now turn to the case when $\phi$ is a compact profile, so $t_n \equiv 0$. In this case, $v_n$ solves the equation
\begin{equation*}
i \d_t v_n + \Delta v_n = |v_n|^{4-\eps_n} v_n\;, \quad v_n(0,\cdot) = \phi, 
\end{equation*}
and $\Psi$ satisfies
\begin{equation*}
i \d_t \Psi + \Delta \Psi = |\Psi|^{4} \Psi\;, \quad \Psi(0,\cdot) = \phi, 
\end{equation*}
and $\Psi_n \equiv \Psi$. The problem is then reduced to the comparison with $v_n$ and $\Psi$ starting with the same fixed initial data $\phi$. 

Let $\delta$ be as in Theorem~\ref{pr:strong_sta} and Proposition~\ref{pr:iteam_eps}. There exists $\tilde{\phi} \in H_{x}^{2}$ such that $\|\tilde{\phi} - \phi\|_{L_{x}^{2}} < \frac{\delta}{2}$. Let $\tilde{\Psi}$ be the solution to the equation
\begin{equation*}
i \d_t \tilde{\Psi} + \Delta \tilde{\Psi} = |\tilde{\Psi}|^{4} \tilde{\Psi}\;, \qquad \tilde{\Psi}(0,\cdot) = \tilde{\phi}. 
\end{equation*}
Then, Proposition~\ref{pr:strong_sta} implies that there exists $C$ depending on $M_0$ only such that
\begin{equation} \label{eq:smooth_initial}
\|\tilde{\Psi} - \Psi\|_{\xX(\iI)} \leq C \delta. 
\end{equation}
It then remains to compare $\tilde{\Psi}$ and $v_n$. Since $\|\tilde{\phi}\|_{H_{x}^{2}}$ depends on $\delta$ and $\phi$ only and that $\phi$ is fixed, we have persistence of regularity in the sense that
\begin{equation*}
\|\tilde{\Psi}\|_{L_{t}^{\infty} L_{x}^{\infty}(\iI)} \leq C_{\delta}. 
\end{equation*}
Now, we re-write the equation for $\tilde{\Psi}$ as
\begin{equation*}
i \d_t \tilde{\Psi} + \Delta \tilde{\Psi} = |\tilde{\Psi}|^{4-\eps_n} \tilde{\Psi} + e_n\;, 
\end{equation*}
where $e_n = \big( |\tilde{\Psi}|^{4} - |\tilde{\Psi}|^{4-\eps_n} \big) \tilde{\Psi}$ satisfies the pointwise bound
\begin{equation*}
|e| \leq C_{\delta} |\tilde{\Psi}|^{5-\eps_n} \cdot \eps_n 
\end{equation*}
for every $(t,x)$. Hence, by Theorem~\ref{th:Dodson}, we have
\begin{equation*}
\|e_n\|_{L_{t}^{1}L_{x}^{2}(\iI)} < C_{\delta} \|\tilde{\Psi}\|_{\xX(\iI)}^{5} \cdot \eps_n < C_{\delta} \cdot \eps_n. 
\end{equation*}
Now, using boundedness of $\|\tilde{\Psi}\|_{\xX(\iI)}$ and Proposition~\ref{pr:strong_sta}, we have
\begin{equation*}
\limsup_{n \rightarrow +\infty} \|v_n - \tilde{\Psi}\|_{\xX(\iI)} \leq C \limsup_{n \rightarrow +\infty} \big( \|e_n\|_{L_{t}^{1}L_{x}^{2}(\iI)} + \|\tilde{\phi}-\phi\|_{L_{x}^{2}} \big) \leq C \delta.  
\end{equation*}
Combining it with \eqref{eq:smooth_initial}, we obtain
\begin{equation*}
\limsup_{n \rightarrow +\infty} \|v_n - \Psi\|_{\xX(\iI)} < C \delta. 
\end{equation*}
Note that this $C$ depends on $M_0$ only. Since it is true for every $\delta>0$, we conclude
\begin{equation*}
\lim_{n \rightarrow +\infty} \|v_n - \Psi\|_{\xX(\iI)} = 0. 
\end{equation*}
This completes the case $\lambda_n \equiv 1$.

\subsubsection{Case 3: $\lambda_n \rightarrow 0$}

The situation when $\phi$ is forward scattering is the same as above. If $\phi$ is backward scattering and $\lim_{n} t_n > 1$, then $\frac{1-t_n}{\lambda_n^2} \rightarrow -\infty$. This implies
\begin{equation*}
\|e^{it\Delta} \phi_n\|_{\xX_2(\iI)} = \|e^{it\Delta} \phi\|_{\xX_2(-\frac{t_n}{\lambda_n^2}, \frac{1-t_n}{\lambda_n^2})} \rightarrow 0. 
\end{equation*}
Hence, similar as above, we have$\|v_n - \Psi_n\|_{\xX(\iI)} \rightarrow 0$, where $\Psi_n = e^{it\Delta} \phi_n$. 

It then remains to consider the situation when $\phi$ is a compact profile, and when $\phi$ is backward scattering and $\lim_{n} t_n \leq 1$. In both situations, $\Psi$ solves
\begin{equation*}
i \d_t \Psi + \Delta \Psi = \mu |\Psi|^{4} \Psi\;, \mu = \lim_{n \rightarrow +\infty} \sqrt{\lambda_n^{\eps_n}}. 
\end{equation*}
We start with the compact profile.

\begin{flushleft}
	\textit{Situation 1.} Compact profile
\end{flushleft}

In this case, $t_n=0$ for all $n$. Let
\begin{equation*}
w_{n}(t,x) = \lambda_{n}^{\frac{1}{2}} v_{n}(\lambda_n^2 t, \lambda_n x), 
\end{equation*}
then $w_n$ solves the equation
\begin{equation*}
i \d_t w_n + \Delta w_n = \lambda_{n}^{\frac{\eps_n}{2}} |w_n|^{4-\eps_n} w_n\;, \qquad w_n(0,\cdot) = \phi\;, 
\end{equation*}
and $\|v_n - \Psi_n\|_{\xX(\iI)} = \|w_n - \Psi\|_{\xX(0,1/\lambda_n^2)}$. We now compare $w_n$ and $\Psi$ in the interval $[0,\frac{1}{\lambda_n^2}]$. 

\begin{lem} \label{le:finite_time_compact}
	For every $T>0$, we have
	\begin{equation*}
	\lim_{n \rightarrow +\infty} \|w_n - \Psi\|_{\xX(0,T)} = 0. 
	\end{equation*}
\end{lem}
\begin{proof}
	Similar as in the case $\lambda_n=1$ and compact profile, we let $\delta$ be sufficiently small (depending on $T$), and let $\tilde{phi} \in H_{x}^{2}$ such that $\|\tilde{\phi} - \phi\|_{L_{x}^{2}} < \delta$. Let $\tilde{\Psi}$ be the solution to
	\begin{equation*}
	i \d_t \tilde{\Psi} + \Delta \Psi = \mu |\tilde{\Psi}|^{4} \tilde{\Psi}\;, \qquad \tilde{\Psi}(0, \cdot) = \tilde{\phi}. 
	\end{equation*}
	The difference $\|\tilde{\Psi} - \Psi\|_{\xX(0,T)}$ is small due to Proposition~\ref{pr:strong_sta}, and the difference $\|w_n - \Psi\|_{\xX(0,T)}$ can be controlled thanks to persistence of regularity. The claim then follows by combining the two and noting that $\delta$ can be arbitrarily small. 
\end{proof}

\begin{lem} \label{le:restart_linear}
	For every $\kappa>0$, there exists $T>0$ such that
	\begin{equation*}
	\|e^{it\Delta} \Psi(T, \cdot)\|_{\xX_2(T,+\infty)} < \kappa\;, \quad \limsup_{n \rightarrow +\infty} \|e^{it\Delta} w_{n}(T,\cdot)\|_{\xX_2(T,+\infty)} < \kappa. 
	\end{equation*}
\end{lem}
\begin{proof}
	We have the bound
	\begin{equation*}
	\|e^{it\Delta} \Psi(T, \cdot)\|_{\xX_2(T,+\infty)} \leq \|\Psi\|_{\xX_2(T,+\infty)} + C \mu \|\Psi\|_{\xX_1(T,+\infty)} \|\Psi\|_{\xX_2(T,+\infty)}. 
	\end{equation*}
	Since $\|\Psi\|_{\xX_2(\RR)} < +\infty$, we have
	\begin{equation*}
	\|\Psi\|_{\xX_2(T,+\infty)} < \kappa
	\end{equation*}
	if $T$ is sufficiently large. The first claim then follows from this observation and the conservation of $L_x^2$ norm for $\Psi$. The second claim follows from the first one and Lemma~\ref{le:finite_time_compact}. 
\end{proof}

By Lemmas~\ref{le:restart_linear} and~\ref{le:small_enhance}, for every $\kappa>0$, there exists $T>0$ such that
\begin{equation*}
\|\Psi - e^{it\Delta}\Psi(T,\cdot)\|_{\xX_2(T,+\infty)} < \kappa, \qquad \limsup_{n \rightarrow +\infty} \|w_n - e^{it\Delta} w_{n}(T,\cdot)\|_{\xX_2(T,+\infty)} < \kappa. 
\end{equation*}
Combining these with Lemma~\ref{le:finite_time_compact}, we get
\begin{equation*}
\limsup_{n \rightarrow +\infty} \|w_n - \Psi\|_{\xX(0,+\infty)} < 2 \kappa. 
\end{equation*}
Since $\kappa$ is arbitrary, this implies
\begin{equation*}
\lim_{n \rightarrow +\infty} \|w_n - \Psi\|_{\xX(0,\frac{1}{\lambda_n^2})} = 0. 
\end{equation*}
This completes the situation when $\phi$ is a compact profile. 

\begin{flushleft}
	\textit{Situation 2.} Backward scattering, and $\lim_n t_n \leq 1$. 
\end{flushleft}

We now consider the final situation when $\phi$ is backward scattering and $\lim_{n} t_n \leq 1$. Let
\begin{equation*}
w_{n}(t,x) = \lambda_{n}^{\frac{1}{2}} v_{n}(\lambda_n^2 t + t_n, \lambda_n x). 
\end{equation*}
Also note that
\begin{equation*}
\Psi(t,x) = \lambda_{n}^{\frac{1}{2}} \Psi_{n}(\lambda_n^2 t + t_n, \lambda_n x), 
\end{equation*}
so we have
\begin{equation*}
\|v_n - \Psi_n\|_{\xX(\iI)} = \|w_n - \Psi\|_{\xX(-\frac{t_n}{\lambda_n^2}, \frac{1-t_n}{\lambda_n^2})}. 
\end{equation*}
We split it into two parts by
\begin{equation*}
\|w_n - \Psi\|_{\xX(-\frac{t_n}{\lambda_n^2}, \frac{1-t_n}{\lambda_n^2})} \leq \|w_n - \Psi\|_{\xX(-\frac{t_n}{\lambda_n^2}, -T)} + \|w_n - \Psi\|_{\xX(-T,+\infty)}, 
\end{equation*}
where $T>0$ is a large time to be specified below. Let $\Phi = e^{it\Delta} \phi$. Since $\|\Phi\|_{\xX_2(\RR)} < +\infty$, for every $\kappa>0$, there exists $T>0$ independent of $n$ such that
\begin{equation*}
\|e^{it\Delta} \Phi(-t_n/\lambda_n^2, \cdot)\|_{\xX_2(-\frac{t_n}{\lambda_n^2}, -T)} = \|\Phi\|_{\xX_2(-\frac{t_n}{\lambda_n^2}, -T)} < \kappa. 
\end{equation*}
Since $w_{n}(-t_n/\lambda_n^2, x) = \Phi(-t_n/\lambda_n^2, x)$, combined with Lemma~\ref{le:small_enhance}, this immediately implies
\begin{equation*}
\|w_n - \Phi\|_{\xX(-\frac{t_n}{\lambda_n^2}, -T)} \leq C \kappa, 
\end{equation*}
where the right hand side is independent of $n$. On the other hand, since
\begin{equation*}
\|\Psi(-t_n/\lambda_n^2, \cdot) - \Phi(-t_n/\lambda_n^2, \cdot)\|_{L_{x}^{2}} \rightarrow 0, 
\end{equation*}
we also have
\begin{equation*}
\limsup_{n \rightarrow +\infty} \|\Psi - \Phi\|_{\xX(-\frac{t_n}{\lambda_n^2}, -T)} \leq C \kappa. 
\end{equation*}
Hence, we get
\begin{equation*}
\limsup_{n \rightarrow +\infty} \|w_n - \Psi\|_{\xX(-\frac{t_n}{\lambda_n^2}, -T)} \leq C \kappa. 
\end{equation*}
The difference $\|w_n - \Psi\|_{\xX(-T, +\infty)}$ can be controlled in essentially the same way as in the compact profile case, where this time one starts from $-T$ rather than $0$, and the data at $-T$ is slightly different but smaller than $C \kappa$. Hence, we will obtain
\begin{equation*}
\limsup_{n \rightarrow +\infty} \|w_n - \Psi\|_{\xX(-\frac{t_n}{\lambda_n^2}, \frac{1-t_n}{\lambda_n^2})} \leq C \kappa. 
\end{equation*}
Since this is true for every $\kappa$, we conclude that
\begin{equation*}
\lim_{n \rightarrow +\infty} \|v_n - \Psi_n\|_{\xX(\iI)} = \lim_{n \rightarrow +\infty} \|w_n - \Psi\|_{\xX(-\frac{t_n}{\lambda_n^2}, \frac{1-t_n}{\lambda_n^2})} = 0. 
\end{equation*}
This completes the proof of Lemma~\ref{le:one_pro_approx}.

\subsection{Concluding the proof}

For every $J \geq 1$, let
\begin{equation*}
V_{J,n} := \sum_{j=1}^{J} v_{j,n} + e^{it\Delta} \omega_{n}^{J}. 
\end{equation*}
Our aim is to show that $V_{J,n}$ is a good approximation to $u_n$ in the sense that
\begin{equation*}
\lim_{J \rightarrow +\infty} \limsup_{n \rightarrow +\infty} \|V_{J,n}-u_{n}\|_{\xX(\iI)} = 0. 
\end{equation*}
To achieve this, we first write down the equation for $V_{J,n}$ as
\begin{equation*}
i \d_t V_{J,n} + \Delta V_{J,n} = \nN^{\eps_n}(V_{J,n}) + e_{J,n}\;, \qquad V_{J,n}(0) = f_{n}\;, 
\end{equation*}
where the error term is given by
\begin{equation} \label{eq:e_Jn}
e_{J,n} = \sum_{j=1}^{J} \nN^{\eps_n}(v_{j,n}) - \nN^{\eps_n}(V_{J,n}). 
\end{equation}
In view of Proposition~\ref{pr:iteam_eps}, we need to show the boundedness of $\|V_{J,n}\|_{\xX_2(\iI)}$ and smallness of $\|e_{J,n}\|_{L_{t}^{1}L_{x}^{2}(\iI)}$. By Lemma~\ref{le:one_pro_approx} and the uniform boundedness of $\Psi_{j,n}$, there exists $C=C(M_0)$ such that
\begin{equation*}
\limsup_{n \rightarrow +\infty} \|v_{j,n}\|_{\xX(\iI)} \leq C(M_0)
\end{equation*}
for all $j$. We also need to control $\|V_{J,n}\|_{\xX_2(\iI)}$ uniform in both $J$ and $n$. 

\begin{lem} \label{le:ortho_vjn}
	For every $j \neq k$, we have
	\begin{equation*}
	\lim_{n \rightarrow +\infty} \|v_{j,n} v_{k,n}\|_{L_{t}^{5/2}L_{x}^{5}(\iI)} = 0. 
	\end{equation*}
\end{lem}
\begin{proof}
	Fix $j, k \geq 1$ with $j \neq k$. Write
	\begin{equation*}
	v_{j,n} = \Psi_{j,n} + (v_{j,n} - \Psi_{j,n}), 
	\end{equation*}
	and similarly for $v_{k,n}$. By \eqref{eq:separate} and Remark~\ref{rem:orthogonality}, we have
	\begin{equation*}
	\lim_{n \rightarrow +\infty} \|\Psi_{j,n} \Psi_{k,n}\|_{L_{t}^{5/2}L_{x}^{5}(\iI)} = 0. 
	\end{equation*}
	By H\"older's inequality and \eqref{eq:asym_approx}, we have
	\begin{equation*}
	\|(v_{j,n}-\Psi_{j,n})(v_{k,n}-\Psi_{k,n})\|_{L_{t}^{5/2}L_{x}^{5}(\iI)} \leq \|v_{j,n}-\Psi_{j,n}\|_{\xX_2(\iI)} \|v_{k,n}-\Psi_{k,n}\|_{\xX_2(\iI)} \rightarrow 0. 
	\end{equation*}
	The cross terms also vanish because of \eqref{eq:asym_approx} and the uniform boundedness of $\|\Psi_{j,n}\|_{\xX(\iI)}$. This completes the proof of the lemma. 
\end{proof}

\begin{lem} \label{le:bd_V_Jn}
	There exists $C=C(M_0)>0$ such that
	\begin{equation*}
	\limsup_{n \rightarrow +\infty} \Big\|\sum_{j=1}^{J} v_{j,n} \Big\|_{\xX_2(\iI)} \leq C(M_0)
	\end{equation*}
	for all $J \geq 1$. 
\end{lem}
\begin{proof}
	Raise the quantity of interest to $5$-th power, we have
	\begin{equation*}
	\begin{split}
	\Big\|\sum_{j=1}^{J} v_{j,n} \Big\|_{\xX_2(\iI)}^{5} &= \int_{0}^{1} \Big( \int_{\RR} \Big| \sum_{j=1}^{J} v_{j,n}(t,x) \Big|^{10} {\rm d} x \Big)^{\frac{1}{2}} {\rm d}t\\
	&\leq \int_{0}^{1} \Big[ \int_{\RR} \Big( \sum_{j=1}^{J} |v_{j,n}(t,x)| \Big)^{10} {\rm d}x \Big]^{\frac{1}{2}} {\rm d}t. 
	\end{split}
	\end{equation*}
	We expand the $10$-th power of the parenthesis by
	\begin{equation*}
	\Big( \sum_{j=1}^{J} |v_j|\Big)^{10} = \sum_{j=1}^{J} |v_{j}|^{10} + \sum |v_{k_1,n} \cdots v_{k_{10},n}|, 
	\end{equation*}
	where the second sum is over the multi-indices $(k_1, \dots, k_{10})$ such that each component runs over $1, \dots, J$ and at least two components are different. Now, integrating out the $x$ variable and further relaxing the bound by exchaning the sum with square root, we get
	\begin{equation*}
	\Big\| \sum_{j=1}^{J} v_{j,n} \Big\|_{\xX_2(\iI)}^{5} \leq \sum_{j=1}^{J} \|v_{j,n}\|_{\xX_2(\iI)}^{5} + \sum \int_{0}^{1} \Big( \int_{\RR} |v_{k_1,n} \cdots v_{k_{10},n}| {\rm d} x \Big)^{\frac{1}{2}} {\rm d}t, 
	\end{equation*}
	where the range of the sum in the second term is the same as above. By Lemma~\ref{le:ortho_vjn} and the uniform-in-$n$ boundedness of $v_{j,n}$ for each $j$, each term in the second sum above vanishes as $n \rightarrow +\infty$. Hence, we have
	\begin{equation*}
	\limsup_{n \rightarrow +\infty} \Big\| \sum_{j=1}^{J} v_{j,n} \Big\|_{\xX_2(\iI)}^{5} \leq \limsup_{n \rightarrow +\infty} \sum_{j=1}^{J}  \|v_{j,n}\|_{\xX_2(\iI)}^{5}. 
	\end{equation*}
	We now need to control the right hand side above independently of $J$. Let $\delta_0$ be the small constant in Lemma~\ref{le:small_mass}. By \eqref{eq:ortho_mass}, we have
	\begin{equation*}
	\sum_{j=1}^{+\infty} \|\phi_j\|_{L_x^2}^2 < 1 + M_0^2, 
	\end{equation*}
	and so there exists $J_0$ such that
	\begin{equation*}
	\|\phi_{j,n}\|_{L_{x}^{2}} = \|\phi_j\|_{L_{x}^{2}} < \delta_0
	\end{equation*}
	for all $j > J_0$. Hence, by Lemma~\ref{le:small_mass}, there exists $C>0$ universal such that
	\begin{equation*}
	\|v_{j,n}\|_{\xX_2(\iI)}^{5} \leq C \|\phi_j\|_{L_{x}^{2}}^{5}
	\end{equation*}
	for all $j > J_0$ and all $n$, which in turn implies (if $J>J_0$)
	\begin{equation} \label{eq:after_J0}
	\sum_{j=J_0+1}^{J} \|v_{j,n}\|_{\xX_2(\iI)}^{5} \leq C \sum_{j=J_0+1}^{+\infty} \|\phi_j\|_{L_{x}^{2}}^{5} \leq C \sum_{j=J_0+1}^{+\infty} \|\phi_j\|_{L_{x}^{2}}^{2} \leq C (1+M_0^2). 
	\end{equation}
	Note that without loss of generality, we can order $\phi_j$ in decreasing $L_x^2$ norm, and hence $J_0$ depends on $M_0$ and $\delta_0$ only. For $j \leq J_0$, by Lemma~\ref{le:one_pro_approx} and the uniform boundedness of $\|\Psi_{j,n}\|_{\xX_2(\iI)}$, we have
	\begin{equation} \label{eq:before_J0}
	\sum_{j=1}^{J_0} \limsup_{n \rightarrow +\infty} \|v_{j,n}\|_{\xX_2(\iI)}^{5} \leq C(M_0,\delta_0). 
	\end{equation}
	The claim then follows by combining \eqref{eq:after_J0} and \eqref{eq:before_J0} and noting that $\delta_0$ is universal. 
\end{proof}

\begin{lem} \label{le:error_ejn}
	The error term $e_{J,n}$ in \eqref{eq:e_Jn} satisfies
	\begin{equation*}
	\lim_{J \rightarrow +\infty} \limsup_{n \rightarrow +\infty} \|e_{J,n}\|_{L_{t}^{1}L_{x}^{2}(\iI)} = 0. 
	\end{equation*}
\end{lem}
\begin{proof}
	We write $e_{J,n} = e_{J,n}^{(1)} + e_{J,n}^{(2)}$, where
	\begin{equation*}
	e_{J,n}^{(1)} = \sum_{j=1}^{J} \big( |v_{j,n}|^{4-\eps_n} - |V_{J,n}|^{4-\eps_n} \big) v_{j,n}\;, \quad e_{J,n}^{(2)} = - |V_{J,n}|^{4-\eps_n} \cdot e^{it\Delta}\omega_{n}^{J}. 
	\end{equation*}
	We first treat $e_{J,n}^{(2)}$. By H\"older's inequality, we have
	\begin{equation*}
	\|e_{J,n}^{(2)}\|_{L_{t}^{1}L_{x}^{2}(\iI)} \leq \|e^{it\Delta}\omega_{n}^{J}\|_{\xX_1(\iI)}^{\frac{\eps_n}{4}} \|e^{it\Delta}\omega_{n}^{J}\|_{\xX_2(\iI)}^{1-\frac{\eps_n}{4}} \|V_{J,n}\|_{\xX_2(\iI)}^{4-\frac{5\eps_n}{4}}. 
	\end{equation*}
	The desired bound for $e_{J,n}^{(2)}$ follows from \eqref{eq:ortho_mass}, \eqref{eq:small_remainder} and Lemma~\ref{le:bd_V_Jn}. As for $e_{J,n}^{(1)}$, for every $j$, we have the pointwise bound
	\begin{equation*}
	\big| \big(|v_{j,n}|^{4-\eps_n} - |V_{J,n}|^{4-\eps_n} \big) v_{j,n} \big| \leq C \big( \sum_{k \neq j} |v_{k,n}| + |e^{it\Delta} \omega_{n}^{J}| \big) \big( |V_{J,n}|^{3-\eps_n} + |v_{j,n}|^{3-\eps_n} \big) |v_{j,n}|, 
	\end{equation*}
	where the sum is taken over all $k \leq J$ with $k \neq j$. Again, by \eqref{eq:ortho_mass}, \eqref{eq:small_remainder} and Lemmas~\ref{le:ortho_vjn} and~\ref{le:bd_V_Jn}, we conclude that
	\begin{equation*}
	\limsup_{n \rightarrow +\infty} \|e_{J,n}^{(1)}\|_{L_{t}^{1}L_{x}^{2}(\iI)} = 0
	\end{equation*}
	for every $J$. The proof is thus complete. 
\end{proof}

We are finally ready to prove the uniform boundedness of $\{u_n\}$. 

\begin{proof} [Proof of Proposition~\ref{pr:main_tech}]
	Let $\delta$ be the small constant as in Proposition~\ref{pr:iteam_eps}. By \eqref{eq:small_remainder} and Lemmas~\ref{le:bd_V_Jn} and~\ref{le:error_ejn}, there exists $J_0>0$ such that for all $J > J_0$ and all sufficiently large $n$ (depending on $J$), we have
	\begin{equation*}
	\|e_{J,n}\|_{L_{t}^{1}L_{x}^{2}} \leq \delta, 
	\end{equation*}
	and
	\begin{equation*}
	\|V_{J,n}\|_{\xX_2(\iI)} \leq \Big\|\sum_{j=1}^{J} v_{j,n}\Big\|_{\xX_2(\iI)} + \|e^{it\Delta} \omega_{n}^{J}\|_{\xX_2(\iI)} \leq C(M_0). 
	\end{equation*}
	Hence, by Proposition~\ref{pr:strong_sta}, we have
	\begin{equation*}
	\|u_n - V_{J,n}\|_{\xX(\iI)} \leq C \delta
	\end{equation*}
	for all $n$ large enough (depending on $J$), and $C$ depends on $M_0$ only. Hence, we deduce
	\begin{equation*}
	\|u_n\|_{\xX_2(\iI)} \leq \|V_{J,n}\|_{\xX_2(\iI)} + C \delta. 
	\end{equation*}
	Taking $n \rightarrow +\infty$ and employing Lemma~\ref{le:bd_V_Jn} gives the uniform boundedness of $\|u_n\|_{\xX_2(\iI)}$. The boundedness of $\|u_n\|_{\xX_1(\iI)}$ follows directly from conservation of mass. Hence the proof is complete. 
\end{proof}

\section{Proof of Theorem~\ref{th:snls_approx}}
\label{sec:stochastic}

\subsection{Proof of the uniform bound~(\ref{eq:ume_uniform_bd})}
\label{sec:stochastic_bd}

The main ingredients to prove \eqref{eq:ume_uniform_bd} is a series of deterministic uniform-in-$(m,\eps)$ boundedness and stability statements, each of which relies on the previous one and the uniform-in-$\eps$ bound in Proposition~\ref{pr:Dodson_eps}. The corresponding statements for $\eps=0$, which rely on Dodson's Theorem~\ref{th:Dodson}, have been treated in detail in \cite{snls_critical}. The proofs of the uniform-in-$(m,\eps)$ statements below are essentially the same as the corresponding ones in \cite{snls_critical}, except that one replaces Theorem~\ref{th:Dodson} by its $\eps$-version Proposition~\ref{pr:Dodson_eps}. Hence, we state the precise boundedness and stability results below, and refer their proofs to the corresponding ones in \cite{snls_critical}. 

Throughout this subsection, we fix the interval $\iI = [a,b]$ with $b-a \leq 1$. Let $w_{m,\eps} \in \xX(\iI)$ be the solution to
\begin{equation} \label{eq:nls_wme}
i \d_t w_{m,\eps} + \Delta w_{m,\eps} = \theta_{m}\big(A + \|w_{m,\eps}\|_{\xX_2(a,t)}^{5} \big) \nN^\eps(w_{m,\eps})\;, \quad w_{m,\eps}(a) \in L_{x}^{2}, 
\end{equation}
and let $v_{m,\eps} \in \xX(\iI)$ and $e \in L_{t}^{1}L_{x}^{2}(\iI)$ satisfy
\begin{equation} \label{eq:nls_vme}
i \d_t v_{m,\eps} + \Delta v_{m,\eps} = \theta_{m} \big(\tilde{A} + \|v_{m,\eps}\|_{\xX_2(a,t)}^{5}\big) \nN^\eps(v_{m,\eps}) + e\;, \quad v_{m,\eps}(a) \in L_{x}^{2}. 
\end{equation}
The following propositions are $(m,\eps)$ versions of the boundedness and stability of mass-critical NLS. All bounds below are uniform in $m$, $\eps$, $A$, $\tilde{A}$ and intervals $\iI$ with $|\iI| \leq 1$.

\begin{prop} \label{pr:Dodson_me}
	Let $w_{m,\eps} \in \xX(\iI)$ be the solution to \eqref{eq:nls_wme}. Then for every $M>0$, there exists $\tilde{D}_M>0$ such that
	\begin{equation*}
	\|w_{m,\eps}\|_{\xX(\iI)} \leq \tilde{D}_M
	\end{equation*}
	as long as $\|w_{m,\eps}\|_{L_{x}^{2}} \leq M$. 
\end{prop}
\begin{proof}
	Same as \cite[Proposition~4.4]{snls_critical}. 
\end{proof}

\begin{prop} \label{pr:strong_sta_me}
	Let $w_{m,\eps}$ be as in \eqref{eq:nls_wme}. Let $v_{m,\eps} \in \xX(\iI)$ and $e \in L_{t}^{1}L_{x}^{2}(\iI)$ satisfy \eqref{eq:nls_vme}. For every $M>0$, there exists $\delta_M, C_M>0$ such that if
	\begin{equation*}
	\|v_{m,\eps}(a)\|_{L_{x}^{2}} \leq M\;, \quad \|v_{m,\eps}(a)-w_{m,\eps}(a)\|_{L_x^2} + \|e\|_{L_{t}^{1}L_{x}^{2}(\iI)} + |A-\tilde{A}| \leq \delta_M, 
	\end{equation*}
	then we have
	\begin{equation*}
	\|v_{m,\eps} - w_{m,\eps}\|_{\xX(\iI)} \leq C_M \Big( \|v_{m,\eps}(a)-w_{m,\eps}(a)\|_{L_x^2} + \|e\|_{L_{t}^{1}L_{x}^{2}(\iI)} + |A-\tilde{A}| \Big). 
	\end{equation*}
	The constants $\delta_M$ and $C_M$ depend on $M$ only. 
\end{prop}
\begin{proof}
	Same as \cite[Proposition~4.5]{snls_critical}. 
\end{proof}

\begin{prop} \label{pr:integral_bd}
	Suppose $u_{m,\eps}, g \in \xX(\iI)$ satisfy $g(a) = 0$ and the integral relation
	\begin{equation*}
	u_{m,\eps}(t) = e^{i(t-a)\Delta}u_{m,\eps}(a) - i \int_{a}^{t} e^{i(t-s)\Delta} \Big( \theta_{m}\big( A + \|u_{m,\eps}\|_{\xX_2(a,s)}^{5} \big) \nN^{\eps}\big(u_{m,\eps}(s)\big) \Big) + g(t)
	\end{equation*}
	on $\iI=[a,b]$, where $\nN^\eps(u) = |u|^{4-\eps}u$. Then, for every $M>0$, there exist $\eta_M, B_M>0$ such that if
	\begin{equation*}
	\|u_{m,\eps}\|_{\xX_1(\iI)} \leq M \qquad \text{and} \qquad \|g\|_{\xX_2(\iI)} \leq \eta_M, 
	\end{equation*}
	then we have
	\begin{equation*}
	\|u_{m,\eps}\|_{\xX(\iI)} \leq B_M. 
	\end{equation*}
	The constants $\eta_M$ and $B_M$ depend on $M$ only. 
\end{prop}
\begin{proof}
	Same as \cite[Proposition~4.6]{snls_critical}. 
\end{proof}

Note that the $u_{m,\eps}$ above is a general space-time function satisfying the assumption in Proposition~\ref{pr:integral_bd}, although as an application, we will use it to get the uniform boundedness of the $u_{m,\eps}$ in Theorem~\ref{th:snls_approx}. 

\begin{proof} [Proof of~\eqref{eq:ume_uniform_bd}]
	The proof follows from Proposition~\ref{pr:integral_bd} in exactly the same way as the proof of \cite[Proposition~1.11]{snls_critical}, where the latter relies on the $\eps=0$ version of Proposition~\ref{pr:integral_bd}. 
\end{proof}

\subsection{Proof of the convergence~(\ref{eq:ume_converge})}
\label{sec:stochastic_converge}

In addition to the uniform bound \eqref{eq:ume_uniform_bd}, we also need a uniqueness result in \cite{BD}. For every $m>0$ and $\eps \in (0,1)$, let
\begin{equation*}
\tau_{m,\eps} := T \wedge \inf \big\{t \geq 0: \|u_{m,\eps}\|_{\xX_2(0,t)}^{5} \geq m \big\}. 
\end{equation*}
We then have the following lemma. 

\begin{lem} [\cite{BD}, Lemma~4.1]
	\label{le:unique_stopping}
	For every $\eps \in (0,1)$ and every $m_1 < m_2$, we have $\tau_{m_1,\eps} \leq \tau_{m_2,\eps}$ almost surely, and $u_{m_1,\eps} = u_{m_2,\eps}$ on $\xX(0,\tau_{m_1,\eps})$ almost surely. 
\end{lem}

We are ready to prove the robustness of approximations. All the norms concerned below are $L_{\omega}^{\rho}\xX(0,T)$ unless otherwise specified, and we omit these norms for notational simplicity. 

Fix $\delta>0$ arbitrary. Since $\|u_m-u\| \rightarrow 0$ (\cite[Theorem~1.3, Section~5]{snls_critical}), there exists $m^{*}>0$ such that
\begin{equation*}
\|u_m - u\| < \frac{\delta}{3}
\end{equation*}
for all $m>m^*$. Also, by the uniform bound \eqref{eq:ume_uniform_bd}, we have
\begin{equation} \label{eq:prob_small}
\Pr \Big( \|u_{m,\eps}\|_{\xX_2(0,T)}^{5} \geq K \Big) \leq \frac{C}{K^{\rho/5}}
\end{equation}
for every $K$, where $C = \|u_{m,\eps}\|$ is independent of $m$, $\eps$ and $K$. For every $m$, $m'$, $\eps$ and $K$, let
\begin{equation*}
\Omega_{m,m',\eps}^{K} = \Big\{ \omega: \|u_{m,\eps}\|_{\xX(0,T)}^{5} \geq K \Big\} \cup \Big\{ \omega: \|u_{m',\eps}\|_{\xX(0,T)}^{5} \geq K \Big\}. 
\end{equation*}
We have
\begin{equation*}
\big\|\1_{\Omega_{m,m',\eps}^{K}}(u_{m,\eps} - u_{m',\eps})\big\| \leq \big( \Pr(\Omega_{m,m',\eps}^{K}) \big)^{\frac{1}{2\rho}} \|u_{m,\eps}-u_{m',\eps}\|_{L_{\omega}^{2\rho}\xX(0,T)} \leq \frac{C}{K^{1/10}}
\end{equation*}
for all $m$, $m'$, $\eps$ and $K$. The first inequality above is H\"older so we have $L_{\omega}^{2\rho}$ instead of $L_{\omega}^{\rho}$ in the middle term. The second inequality follows from \eqref{eq:prob_small} and the uniform bound \eqref{eq:ume_uniform_bd} with $\rho$ replaced by $2\rho$. Hence, there exists $K^*$ such that
\begin{equation*}
\big\|\1_{\Omega_{m,m',\eps}^{K}}(u_{m,\eps} - u_{m',\eps})\big\| < \frac{\delta}{3}
\end{equation*}
as long as $K \geq K^*$, and this is true for all $m$, $m'$ and $\eps$. Now if $m, m' \geq K^{*}+1$, then
\begin{equation*}
\|u_{m,\eps} - u_{m',\eps}\| \leq \|\1_{\Omega_{m,m',\eps}^{K}}(u_{m,\eps} - u_{m',\eps})\| + \|\1_{(\Omega_{m,m',\eps}^{K})^{c}}(u_{m,\eps} - u_{m',\eps})\| < \frac{\delta}{3}. 
\end{equation*}
This is because by Lemma~\ref{le:unique_stopping}, $u_{m} = u_{m'}$ before $\tau_{K^*}$ if $m,m' \geq K^*+1$, and hence the second term above is $0$. Now, we let
\begin{equation*}
m_0 = m^* \wedge (K^*+1), 
\end{equation*}
and write
\begin{equation*}
\|u_{m,\eps}-u\| \leq \|u_{m,\eps}-u_{m_0,\eps}\| + \|u_{m_0,\eps} - u_{m_0}\| + \|u_{m_0}-u\|. 
\end{equation*}
The third term is smaller than $\frac{\delta}{3}$ since $m_0 > m^*$. The first term is also smaller than $\frac{\delta}{3}$ if $m>m_0>K^*+1$. Finally, for the middle term, by the convergence $\|u_{m_0,\eps}-u_{m_0}\| \rightarrow 0$ (\cite[Proposition~1.10]{snls_critical}), we can choose $\eps_0 = \eps_0 (\delta,m_0)$ such that the middle term is also smaller than $\frac{\delta}{3}$ as long as $\eps<\eps_0$. Since $m_0$ depends on $\delta$ only, so does $\eps_0$. This completes the proof of Theorem~\ref{th:snls_approx}.

\bibliographystyle{Martin}
\bibliography{Refs}

\newcommand{\etalchar}[1]{$^{#1}$}
\begin{thebibliography}{CKS{\etalchar{+}}08}
\expandafter\ifx\csname url\endcsname\relax
  \def\url#1{\texttt{#1}}\fi
\expandafter\ifx\csname urlprefix\endcsname\relax\def\urlprefix{URL }\fi
\expandafter\ifx\csname href\endcsname\relax
  \def\href#1#2{#2}\fi
\expandafter\ifx\csname burlalt\endcsname\relax
  \def\burlalt#1#2{\href{#2}{\texttt{#1}}}\fi

\bibitem[BV07]{begout2007mass}
\textsc{P.~B{\'e}gout} and \textsc{A.~Vargas}.
\newblock Mass concentration phenomena for the {$L^2$}-critical nonlinear
  {S}chr{\"o}dinger equation.
\newblock \emph{Trans. Amer. Math. Soc.} \textbf{359}, no.~11, (2007),
  5257--5282.

\bibitem[Caz03]{cazenave2003semilinear}
\textsc{T.~Cazenave}.
\newblock \emph{Semilinear {S}chr{\"o}dinger equations}.
\newblock Courant Lecture Notes in Mathematics. American Mathematical Society,
  2003.

\bibitem[CK07]{carles2007role}
\textsc{R.~Carles} and \textsc{S.~Keraani}.
\newblock On the role of quadratic oscillations in nonlinear {S}chr{\"o}dinger
  equations ii. the {$L^2$}-critical case.
\newblock \emph{Trans. Amer. Math. Soc.} \textbf{359}, no.~1, (2007), 33--62.

\bibitem[CKS{\etalchar{+}}08]{colliander2008global}
\textsc{J.~Colliander}, \textsc{M.~Keel}, \textsc{G.~Staffilani},
  \textsc{H.~Takaoka}, and \textsc{T.~Tao}.
\newblock Global well-posedness and scattering for the energy-critical
  nonlinear {S}chr{\"o}dinger equation in $\mathbb{R}^3$.
\newblock \emph{Ann. Math.} \textbf{167}, no.~3, (2008), 767--865.

\bibitem[CW89]{cazenave1989some}
\textsc{T.~Cazenave} and \textsc{F.~B. Weissler}.
\newblock Some remarks on the nonlinear {S}chr{\"o}dinger equation in the
  critical case.
\newblock In \emph{Nonlinear Semigroups, Partial Differential Equations and
  Attractors},  18--29. Springer, 1989.

\bibitem[dBD99]{BD}
\textsc{A.~de~Bouard} and \textsc{A.~Debussche}.
\newblock A stochastic nonlinear {S}chr\"{o}dinger equation with multiplicative
  noise.
\newblock \emph{Comm. Math. Phys.} \textbf{205}, no.~1, (1999), 161--181.

\bibitem[Dod16]{dodson2016global}
\textsc{B.~Dodson}.
\newblock Global well-posedness and scattering for the defocusing,
  {$L^2$}-critical, nonlinear {S}chr{\"o}dinger equation when $d= 1$.
\newblock \emph{Amer. J. Math.} \textbf{138}, no.~2, (2016), 531--569.

\bibitem[FX18]{snls_critical}
\textsc{C.~Fan} and \textsc{W.~Xu}.
\newblock Global well-posedness for the defocusing mass-critical stochastic
  nonlinear {S}chr\"{o}dinger equation on $\mathbb{R}$ at {$L^2$} regularity.
\newblock \emph{ArXiv e-prints} (2018).
\newblock \burlalt{arXiv:1810.07925}{http://arxiv.org/abs/1810.07925}.

\bibitem[KT98]{keel1998endpoint}
\textsc{M.~Keel} and \textsc{T.~Tao}.
\newblock Endpoint {S}trichartz estimates.
\newblock \emph{Amer. J. Math.} \textbf{120}, no.~5, (1998), 955--980.

\bibitem[MV98]{merle1998compactness}
\textsc{F.~Merle} and \textsc{L.~Vega}.
\newblock Compactness at blow-up time for {$L^{2}$} solutions of the critical
  nonlinear {S}chr{\"o}dinger equation in 2{D}.
\newblock \emph{Int. Math. Res. Not. IMRN} \textbf{1998}, no.~8, (1998),
  399--425.

\bibitem[Tao06]{tao2006nonlinear}
\textsc{T.~Tao}.
\newblock \emph{Nonlinear dispersive equations: local and global analysis}.
\newblock CBMS Regional Conference Series in Mathematics. American Mathematical
  Society, 2006.

\bibitem[TVZ08]{tao2008minimal}
\textsc{T.~Tao}, \textsc{M.~Visan}, and \textsc{X.~Zhang}.
\newblock Minimal-mass blowup solutions of the mass-critical {NLS}.
\newblock \emph{Forum Math.} \textbf{20}, no.~5, (2008), 881--919.

\end{thebibliography}

\end{document}